\documentclass  [article]{amsart}

\usepackage{latexsym, amsfonts, mathrsfs, amsmath, amssymb, amscd, epsfig, amsrefs}
\usepackage{color}
\usepackage{esint}
\allowdisplaybreaks[0]

\newtheorem{theorem}{Theorem}[section]
\newtheorem{cor}[theorem]{Corollary}
\newtheorem{lemma}[theorem]{Lemma}

\newtheorem{remark}[theorem]{Remark}
\newtheorem{prop}[theorem]{Proposition}
\newtheorem{Add in proof}{Add in proof}[section]
\numberwithin{equation}{section}

\def\R{{\Bbb R}} 

 \def\p#1#2{\dfrac{\partial
#1}{\partial#2}}
\def \p{\partial}

 \def\a{\alpha}
\def\intave#1{-\kern-10.7pt\int_{\,#1}}
\def\b{\beta}
\def\D{\nabla}
\def\g{\gamma}

\def\<{\langle}
\def\>{\rangle}
\def\({\left(}
\def\){\right)}
\def\esssup{\operatornamewithlimits{ess\,sup}}
\def\limsup{\operatornamewithlimits{lim\,sup}}
\def\intave#1{-\kern-10.7pt\int_{\,#1}}

\begin{document}
\title{Well-posedness of the Ericksen-Leslie system for the Oseen-Frank model in $L^3_{uloc}(\R^3)$}

\author{Min-Chun Hong and  Yu Mei}

\address{Min-Chun Hong, Department of Mathematics, The University of Queensland\\
Brisbane, QLD 4072, Australia}

\address{Yu Mei, Department of Mathematics, The University of Queensland\\
Brisbane, QLD 4072, Australia}

\begin{abstract}
{We investigate the Ericksen-Leslie system for the Oseen-Frank model with unequal Frank elastic constants in $\R^3$. To generalize the result of Hineman-Wang \cite{HW}, we   prove existence of  solutions to the  Ericksen-Leslie system    with initial data  having small $L^3_{uloc}$-norm.  In particular, we use a new idea to obtain a local $L^3$-estimate through interpolation inequalities and  a covering argument, which is different from the one in \cite{HW}. Moreover, for uniqueness of solutions,   we  find a new way to remove  the restriction on the Frank elastic constants    by using  the rotation invariant property of the Oseen-Frank density. We  combine this with a method  of Li-Titi-Xin  \cite{LTX} to prove  uniqueness of the $L^3_{uloc}$-solutions of the Ericksen-Leslie system   assuming that the initial data has  a finite energy.   }
     \end{abstract}

 \maketitle

\pagestyle{myheadings} \markright {The Ericksen-Leslie system }

\section{\bf Introduction}

 Liquid crystals are states of matter intermediate between solid crystals and normal isotropic liquids. One of the most common liquid crystal phases is nematic. It is composed of rod like molecules which exhibit optically distinguished local directions, unlike a liquid, but lacking the lattice structure of a solid. The macroscopic description of nematic liquid crystal flows is based on the continuum model for configuration of crystals and conservation laws for anisotropic fluids. In their seminal works, Oseen \cite{Os} and Frank \cite{Fr} established  variational theory for the static configurations of liquid crystals through seeking the director vector $u\in H^1(\Omega, S^2)$ to minimize the (Oseen-Frank) elastic distortion energy
 \begin{align}\label{mini}
 E(u)=\int_{\Omega}W(u,\D u)dx,
 \end{align}
 where the free energy density $W(u,\D u)$ is of the form
 \begin{align}\label{O-F-density}
 W(u,\D u)=&k_1(\text{div } u)^2+k_2 (u\cdot\text{curl }u)^2+k_3
 |u\times \text{curl } u|^2\\
 &+(k_2+k_4)\left(\text {tr} (\nabla u)^2-(\text {div
 }u)^2\right).\nonumber
 \end{align}
 Here $k_1,k_2,k_3,k_4$ are the Frank elastic constants, which are usually assumed to satisfy Ericksen's inequalities (\cite{ Er-2}, \cite{BE})
 \begin{equation}\label{k-con}
 k_1> 0, \,k_2 >|k_4|,\,k_3>0, \,2k_1\geq k_2+k_4.
 \end{equation}
The first three terms on the right hand of \eqref{O-F-density} are associated with the splay, twist, bend characteristic deformations respectively,  and the fourth term there is responsible to the surface free energy term, which is a null Lagrangian, so that the integral
 $$(k_2+k_4)\int_{\Omega}\left(\text {tr} (\nabla u)^2-(\text {div
 }u)^2\right)dx$$
 depends only on the boundary value of $u$. Based on a generalization of the static Oseen-Frank theory, Ericksen \cite{Er} described  dynamic behaviors of nematic liquid crystals by using conservation laws from continuum fluid mechanics around the 1960s. Later, Leslie \cite{Le} proposed some constitutive equations for nematic liquid crystals and completed  the hydrodynamic theory of liquid crystals. The Ericksen-Leslie theory now is one of the most successful theories for  modelling the nematic liquid crystal flow.

  In this paper, we investigate the  Ericksen-Leslie system for the Oseen-Frank model with unequal elastic constants in $\R^3$.   Let $v=(v^1,v^2,v^3)$ be the velocity field of the fluid, $u=(u^1,u^2,u^3)$  the unit director vector representing the preferred direction of molecular alignment  with its density $W(u,\nabla u)$ and $p$   the pressure.    The  Ericksen-Leslie system is:
 \begin{align}
 &\p_t v^i+(v\cdot\D) v^i+\D_i p=\Delta v^i-\D_j(\D_i u^k W_{p_j^k}),\label{O-F-1}\\
 &\D\cdot v=0,\label{O-F-2}\\
 &\p_tu^i+(v\cdot \D)u^i=\D_\alpha\left(W_{p_\alpha^i}-u^ku^iW_{p_\alpha^k}\right)-W_{u^i}+W_{u^k}u^iu^k\label{O-F-3}\\
 &\qquad\qquad\qquad\qquad\qquad\qquad+W_{p_\alpha^k}\D_\alpha u^ku^i+W_{p_\alpha^k}u^k\D_\alpha u^i.\nonumber
 \end{align}
 When  $k_1=k_2=k_3=1$ and $k_4=0$, $W(u,\D u)=|\D u|^2$ in \eqref{O-F-density} and  the system \eqref{O-F-1}-\eqref{O-F-3}, which is now called the simplified Ericksen-Leslie system, is  a system of the Navier-Stokes equations
coupled with the harmonic map flow.

 The global existence  of
weak solutions to the  Ericksen-Leslie system  is a long-standing open
problem since 1960s. In order to solve this challenging problem, Lin and Liu \cite{LL1,LL2}    investigated the approximate  Ericksen-Leslie system by   the
Ginzburg-Laudau functional  and
proved the global existence of  the
classical solution of the approximate  Ericksen-Leslie system   in 2D and   the weak solution
of the same system  in 3D. However, Lin and Liu \cite {LL2} were not able to show whether the
limit of solutions $(v_{\varepsilon},u_{\varepsilon})$ of the approximate  Ericksen-Leslie systems as $\varepsilon\to 0$  satisfies the
Ericksen-Leslie system, so  there is an open problem whether the limit of solutions of the approximation system  is a solution of
 the Ericksen-Leslie system.  Based  on the    result of Struwe \cite{St} on harmonic maps, Lin-Lin-Wang \cite{LLW} and the first author \cite{Ho3} independently proved global existence of weak solutions, which is smooth away from at most finitely many singular times, to the  Ericksen-Leslie system   in   dimension two. Concerning the effect of  Leslie stress tensor, Huang-Lin-Wang \cite{HLW} also obtained global existence and regularity of weak solutions in $\R^2$.  The uniqueness of weak solutions to the simplified Ericksen-Leslie system in $\R^2$ was proved by Lin-Wang \cite{LW2} and Xu-Zhang \cite{XZ}. However, the question  on the global existence to the
Ericksen-Leslie system in  dimension three remains unresolved. In view of the study on the Navier-Stokes
equations, Huang-Wang \cite{HW2} established the local well-posedness and    blow-up criteria  for  the simplified Ericksen-Leslie system. This result was generalized by Wang-Zhang-Zhang \cite{WZZ}  to the simplified Ericksen-Leslie system with Leslie stress tensor.

 In order to prove   global existence of weak solutions to the Ericksen-Leslie system, it is interesting to prove  existence of solutions in $\R^3$ with rough initial data $(v_0, u_0)$. Indeed, Lin-Wang \cite{LW} proved global existence of weak solutions to the simplified Ericksen-Leslie system in dimension three with rough initial data $(v_0,u_0)\in L^2(\R^3)\times H^1(\R^3; S^2)$ with $\D\cdot v_0=0$ and  $u_0$ satisfying the hemisphere condition. On the other hand,  local and global existence of solutions to the Navier-Stokes equations with rough initial data has been studied by many authors (\cite{Cannone,Kato,Planchon,K-T,LR}). In particular, Kato \cite{Kato}  proved local and global existence of the Navier-Stokes equations with initial data in $L^3(\R^3)$.  Koch-Tartaru \cite{K-T} investigated local and global existence of the Navier-Stokes equations   with rough initial data in $BMO^{-1}$.  Inspired  by these results on the Navier-Stokes equations,   Wang \cite{W} established the well-posedness for the simplified  Ericksen-Leslie system with rough initial data $(v_0,u_0)$ in $BMO^{-1}\times BMO$.  Hineman-Wang \cite{HW} obtained the local well-posedness to the simplified system  with the uniformly locally $L^3$-integrable data of $(v_0,\D u_0)$ of small norm.  Later, high order space-time regularities of the solution in \cite{W} were established in \cite{DW,Lin}. Very recently, Hieber  et al. \cite{HNJS} proved existence of strong solutions to the simplified Ericksen-Leslie system in a bounded domain by using the maximal $L_p$-regularity theory for abstract quasi-linear parabolic problems.

 When the Frank elastic constants $k_i$ in \eqref{O-F-density} are unequal, the study of the Ericksen-Leslie system becomes very complicated. This is because, even in the static theory, the Euler-Lagrangian equation associated to the energy functional \eqref{mini} is not standard elliptic for all possible $k_i$ and the energy minimizer may not satisfy the energy monotonicity inequality, which holds for minimizing harmonic maps. In fact, Hardt-Kinderlehrer-Lin \cite{HLK} proved  existence and partial regularity of minimizers of the Oseen-Frank energy functional with unequal $k_i$, which is different from  minimizing harmonic maps (see also  \cite{Ho1}). For further discussion on the static equilibrium problem with unequal elastic constants, we refer to  a recent survey paper by Ball \cite{Ball}.  As to the Ericksen-Leslie system with unequal $k_i$, the equation (\ref {O-F-3}) for the director vector $u$ is  not   a standard parabolic equation. In particular, the maximum principle for $|u|$ is invalid in general. Due to these   additional difficulties,   the study of the Ericksen-Leslie system with unequal $k_i$ is more challenging than the study of the simplified  system. Hong-Xin \cite{HX} proved global existence of weak solutions with regularities except for at most finitely many singular times to the Ericksen-Leslie system \eqref{O-F-1}-\eqref{O-F-3} in $\R^2$. Later, Wang-Wang \cite{WW} generalized this result to the case of the general Ericksen-Leslie system with Leslie stress tensor under certain constraints on the Leslie coefficients. The uniqueness of these global weak solutions in $\R^2$ was proved by Wang-Wang-Zhang \cite {WWZ} and Li-Titi-Xin \cite{LTX} independently. For the three dimensional problem, Hong-Li-Xin \cite{HLX} established the local well-posedness and blow-up criteria of  strong solutions to the Ericksen-Leslie system \eqref{O-F-1}-\eqref{O-F-3} with initial data $(v_0, u_0)\in  H^1(\R^3)\times H^2_b(\R^3, S^2)$ as well as the convergence  of the Ginzburg-Landau approximation system  in $\R^3$. Wang-Wang \cite{WW} proved the local well-posedness of strong solutions to the general Ericksen-Leslie system with Leslie stress tensor for an initial data $(v_0,  u_0)$ satisfying $(v_0,\D u_0)\in H^{2s}(\R^3)\times H^{2s}(\R^3)$ for $s\geq 2$.

 The aim of this paper is to establish the well-posedness of \eqref{O-F-1}-\eqref{O-F-3} in $\R^3$ with rough initial data $(v_0,u_0)$. Motivated by  the work of Hineman-Wang \cite{HW}, we will investigate this problem for initial data with $(v_0,\D u_0)$ in the  uniformly locally $L^3$-integrable space $L^3_{uloc}(\R^3)$. It should be remarked that  $L^3_{uloc}(\R^3)$ is a critical space for the Ericksen-Leslie system in the sense that under the scaling $(v^\lambda, u^\lambda, p^\lambda)(x,t)=\left(\lambda v(\lambda x,\lambda^2 t), u(\lambda x,\lambda^2 t), \lambda^2p(\lambda x,\lambda^2t)\right)$, one has
\[\int_{B_{1/\lambda}}|v^\lambda |^3+ |\nabla u^\lambda|^3\,dx=\int_{B_1}|v |^3+ |\nabla u|^3\,dx\]
and the Ericksen-Leslie system is invariant. Before stating our main results, let us first recall the definition of the  space $L^3_{uloc}(\R^3)$.

\smallskip
\noindent{\bf Definition.}
 		{\it The uniformly locally $L^3$-integrable space is defined to be the set of all functions $f\in L_{loc}^3(\R^3)$ such that
 		$$ \|f\|_{L^3_{R}(\R^3)}:=\sup_{x\in \R^n}\|f\|_{L^3(B_{R}(x))}<+\infty $$ for some $R>0$.
 }\smallskip

One of our main results in this paper is the following  existence theorem:
 \begin{theorem} \label{thm1}
 There exist $\varepsilon_0>0$ and $\sigma>0$ such that if $(v_0,u_0):\R^3\rightarrow \R^3\times S^2$  satisfying
 \begin{align}\label{Ini}
 \|(v_0, \D u_0)\|_{L_{R_0}^3(\R^3)}\leq \varepsilon_0, \quad\lim\limits_{x\to\infty}\fint_{B_1(x)}|u_0(y)-b|^2dy=0,\quad\D\cdot v_0=0
 \end{align}
 for some constant unit vector $b$ and $R_0>0$, then there exist a maximal  time $T^*\geq\sigma R_0^2$ and a solution $(v,u):\R^3\times[0,T^*)\rightarrow \R^3\times S^2$ to \eqref{O-F-1}-\eqref{O-F-3} with initial data $(v_0,u_0)$ such that
 \begin{itemize}
 	\item [(i)] $ (v, \D u)\in C_\ast([0,T^*), L_{uloc}^3(\R^3))$ and $\| (v, \D u)(t)\|_{L^\infty(0,\sigma R_0^2; L_{R_0}^3(\R^3))}\leq C\varepsilon_0$ for  some absolute constant $C$.
 	\item[(ii)] $(v, u)$ is regular on $(0,T^\ast)$, i.e. $(v, u)\in C^\infty((0,T^*),\R^3\times S^2)$.
 	\item[(iii)] The maximal existence time $T^*$ can be characterized by the condition that for any $0<R<+\infty$, there is a $x_i\in\R^3$ such that
 	\begin{equation}\label{max-time}
 	\limsup_{t\rightarrow T^*}\|(v,\D u)\|_{L^3(B_R(x_i))}>\varepsilon_0.
 	\end{equation}
 \end{itemize}
 Here the notation $f\in C_\ast([0,T), L^3_{uloc}(\R^3))$ means that $t\mapsto f(t)$ is continuous in $L^3_{R_0}$- norm for any $t\in(0,T)$ and $f(t)\rightarrow f(0)$ in $L^3_{loc}(\R^3)$ as $t\rightarrow 0$.
 \end{theorem}
 Theorem \ref{thm1} generalizes the existence result  of Hineman-Wang \cite{HW} for the simplified   Ericksen-Leslie system.
 We would like to emphasize that    our proof of Theorem \ref{thm1} is  different from the one in \cite{HW}. For the simplified   Ericksen-Leslie system (i.e. $k_1=k_2=k_3=1,k_4=0$), the principle term of the right hand of  \eqref{O-F-3} is $\Delta u$, so  one has the following property:
  \[\int_{B_{R}}\left <\Delta  \nabla u , |\nabla u| \nabla u\right >\phi^2\,dx  =-\int_{B_{R}} |\nabla u|\,[|\nabla u^2|^2 + |\nabla |\nabla u||^2 ]\, \phi^2\,dx+\mbox{low  term}.\]
   Using this property, Hineman-Wang \cite{HW} got the key local decay $L^3$-estimate for the simplified   Ericksen-Leslie system under small $L^3_{uloc}$-norm of $(v_0, \D u_0)$. However,
 when $k_1$, $k_2$, $k_3$ are unequal, the principle term of the right hand of  \eqref{O-F-3} is $\D_\alpha (W_{p_\alpha^i}(u,\nabla u))$, which is completely different from $\Delta u^i$, so we cannot expect to follow the same idea in \cite{HW} to get the $L^3$-energy-dissipative structure. To overcome this difficulty, we present a new method  to derive the local $L^3$-estimate for the rough initial data $(v_0, u_0)$ satisfying \eqref{Ini}. More precisely, we observe that the $L^2$-energy-dissipative structure for the nonlinear term in \eqref{O-F-3} always holds true due to $|u|=1$. Hence, under the assumption that the local $L^3$-energy of $(v,\D u)$ is small, we first derive  two key estimates for $\frac{1}{R}\int_{B_R(x_0)}|v|^2+|\D u|^2dx$ and $R\int_{B_R(x_0)}|\D v|^2+|\D^2 u|^2dx$ (away from the initial time) by standard energy methods and covering arguments. Furthermore, we  verify the smallness assumption of local $L^3$-energy of $(v,\D u)$ based on these local estimates and the following interpolation inequality (c.f. \cite{CKN})
 \begin{align*}
 \int_{B_r}|f|^3dx\leq C(\frac{1}{r}\int_{B_r}|f|^2dx)^{\frac{3}{4}}(r\int_{B_r}|\nabla f|^2dx)^\frac{3}{4}+C(\frac{1}{r}\int_{B_r}|f|^2dx)^{\frac{3}{2}}.
 \end{align*}
 In this process, we need three different kinds of space-time estimates for the pressure in $L^\frac{3}{2}$, $L^2$ and $L^3$. The proof of  these key  estimates on the pressure is based on the well-known Calderon-Zygmund  theory (c.f \cite{Stein}) and covering arguments. Then, we obtain a uniform estimate in $L^3_{uloc}(\R^3)$ of $(v,\D u)$ and uniform lower bounds of the existence time by  complicated covering arguments (see Section 4 for more details). Moreover, with the uniform estimates of the small local  $L^3$-norm of $(v,\nabla u)$ in hand, we derive local higher   regularity  estimates such as
 $\int_{B}|\nabla^{k+1} u|^2 + |\nabla^k v|^2\,dx$ for any $k\geq 1$ based on a standard energy method and an induction argument. Such kind of local higher  regularity estimates was established by  Huang-Lin-Wang \cite{HLW}  for the simplified  Ericksen-Leslie system with Leslie stress tensor in $\R^2$  and   similar  global higher regularity estimates were obtained by Wang-Wang \cite{WW} for the general Ericksen-Leslie system in $\R^2$. Finally, the characterization of maximal time is proved by contradiction.

 As a consequence of Theorem \ref{thm1}, we have the following global existence result:
 \begin{cor}\label{cor}
 	Let $(v_0,u_0)$ be the initial data in $L^3(\R^3,\R^3)\times \dot{W}^{1,3}(\R^3,S^2)$ with $\D\cdot v_0=0$ and $u_0(x)\rightarrow b$ as $|x|\rightarrow \infty$ for some unit vector $b$. Then there exists a positive constant $\varepsilon_0$ such that if
 	\begin{equation}\label{Ini-2}
 	\|v_0\|_{L^3(\R^3)}+\|u_0\|_{\dot{W}^{1,3}(\R^3)}\leq \varepsilon_0,
 	\end{equation}
 	then  the system \eqref{O-F-1}-\eqref{O-F-3} with the initial data $(v_0,u_0)$ has a solution $(v,u)$ such that
 	\begin{equation*}
 	(v,u)\in C^\infty(\R^3\times(0,+\infty))\cap C([0,+\infty),L^3(\R^3,\R^3)\times \dot{W}^{1,3}(\R^3,S^2)).
 	\end{equation*}
 	Here $\dot{W}^{1,3}(\R^3)$ denotes the standard homogeneous Sobolev spaces.
 \end{cor}

In the second part of this paper, we  prove the uniqueness of $L^3_{uloc}$-solutions in Theorem \ref{thm1}. We set
\begin{equation*}
H_b^1(\R^3;S^2)=\{u:u-b\in H^1(\R^3,\R^3),\,\, \text{and}\,\, |u|=1\,\,\text{a.e. in}\,\,\R^3\}
\end{equation*}
for some $b\in S^2$. Then, our uniqueness theorem is stated as follows:
\begin{theorem}\label{thm2}
 	Let $(v_0,u_0)$ be the initial data satisfying the assumption \eqref{Ini} in Theorem \ref{thm1}. Moreover, assume that $(v_0,u_0)$ satisfy
 	\begin{equation}\label{Ini-3}
 	(v_0,u_0)\in L^2(\R^3;\R^3)\times H^1_b(\R^3;S^2).
 	\end{equation} Then, the solution to \eqref{O-F-1}-\eqref{O-F-3} with initial data $(v_0, u_0)$ is unique.
 \end{theorem}

To  prove the uniqueness of weak solutions to the Ericksen-Leslie system, we use the  idea of Li-Titi-Xin \cite{LTX}   by introducing  the  vector field $(-\Delta+I)^{-1}v$,  but  they needed   a condition of  $k_1,k_2,k_3$ being close to a positive constant to handle the terms involving $W(u,\nabla u)$. In this paper, we are able to remove the restriction of the Frank constants $k_i$  in the $L^2$-norm estimate of the difference of two solutions to the director equation obtained in \cite{LTX} by using the rotation invariant property of $W(u,\D u)$. The similar idea was also used by the first author \cite{Ho1}  to prove the partial regularity of weak solutions to the static Oseen-Frank system. Furthermore, the idea of proving Theorem \ref{thm2} can also be used to prove the uniqueness of weak solutions to \eqref{O-F-1}-\eqref{O-F-3} in $\R^2$, which gives an affirmative answer to the uniqueness question in \cite{HX}.  Comparing with the uniqueness result of Wang-Wang-Zhang \cite{WWZ} in $\R^2$ by using the Littlewood-Paley theory, our method in the $L^2$ framework seems much simpler than theirs.  It should be remarked that  Hineman-Wang \cite{HW} proved the uniqueness of $L^3_{ulco}$-solutions to the simplified Ericksen-Leslie system without the initial finite energy assumption \eqref{Ini-3}. However, it seems difficult for us to apply the same idea, which relies on the properties of heat kernel, to the system \eqref{O-F-1}-\eqref{O-F-3} with unequal  $k_i$.

\begin{remark}
	 For the general Ericksen-Leslie system  with Leslie stress tensor, the energy-dissipation law is complicated and the local pressure estimates  even  for the simplified case require much more technique details, so we will investigate this problem in a forthcoming paper.
\end{remark}
\

The rest of this paper is organized as follows. In Section 2, we prove some key local a prior estimates  to the Ericksen-Leslie system \eqref{O-F-1}-\eqref{O-F-3}.  In Section 3, higher regularity estimates are derived under a smallness assumption. In Section 4, we prove existence of $L^3_{uloc}$-solutions and characterize the maximal existence time. Finally, in Section 5, we prove the uniqueness of $L^3_{uloc}$-solutions.

\section{\bf Key local estimates}\label{sec2}

In this section, we derive a priori estimates for smooth solutions to the Ericksen-Leslie system \eqref{O-F-1}-\eqref{O-F-3}.
Under the physical constrain   condition \eqref{k-con}, it is clear that there exists a $a=\min\{k_2, k_3, k_2+k_4\}>0$ such that
\begin{equation}\label{uni-ell}
W(z,p)\geq a|p|^2,\quad W_{p_\alpha^ip_\beta^j}(z,p)\xi_\alpha^i\xi_\beta^j\geq a|\xi|^2,\quad\forall z\in\mathbb R^3, p,\xi\in\mathbb M^{3\times3},
\end{equation}
which follows from the identities
$$|\D u|^2=\text{tr}(\D u)^2+|\text{curl }u|^2,\quad |\text{curl } u|^2=(u\cdot \text{curl } u)^2+|u\times\text{curl }u|^2.$$
Moreover, $W(u,p)$ is quadratic in both $u$ and $p$ which yields that
\begin{align}\label{qua-est}
&|W(u,\nabla u)|\leq C|u|^2|\nabla u|^2,\quad |W_{u^i}(u,\nabla u)|\leq C|u||\nabla u|^2,\nonumber\\
&|W_{u^iu^j}(u,\nabla u)|\leq C|\nabla u|^2,\quad  |W_{p_\alpha^i}(u,\nabla u)|\leq C|u|^2|\nabla u|,\\
&|W_{p_\alpha^i p_\beta^j}(u,\nabla u)|\leq C|u|^2,\quad |W_{u^ip_\beta^j}(u,\nabla u)|\leq C|u||\nabla u|.\nonumber
\end{align}

In order to obtain the a priori estimates, we start from the following local energy inequality.
\begin{lemma}\label{key-lem1}
	Let $(u, v)$ be a smooth solution to \eqref{O-F-1}-\eqref{O-F-3} in $\mathbb R^3\times(0, T)$. Then for any $\phi\in C_0^\infty(\R^3)$, it holds that, for any $t\in(0,T)$,
	\begin{align}\label{loc-est1}
	&\int_{\R^3}(|v(t)|^2+|\D u(t)|^2)\phi^2dx+\int_{0}^{t}\int_{\R^3}(|\D v|^2+a|\D^2u|^2)\phi^2dxds\nonumber\\
	&\leq \int_{\R^3}(|v_0|^2+|\D u_0|^2)\phi^2dx+4\int_{0}^{t}\int_{\R^3}(p-c(s))v\cdot\D\phi \phi dxds\\
	&\quad +C\int_{0}^{t}\int_{\R^3}(|v|^4+|\D u|^4)\phi^2dxds+C\int_{0}^{t}\int_{\R^3}(|v|^2+|\D u|^2)|\D\phi|^2dxds,\nonumber
	\end{align}
	where $c(t)$ is an arbitrary temporal function and $C$ is an absolute constant.
\end{lemma}
\begin{proof}
Multiplying \eqref{O-F-1} by $v^i\phi^2$, integrating  by parts, and using \eqref{O-F-2}, \eqref{qua-est} and  Young's inequality, we obtain
\begin{align}\label{2.2}
&\frac{1}{2}\frac{d}{dt}\int_{\mathbb R^3}|v|^2\phi^2dx+\int_{\R^3}|\D v|^2\phi^2dx\nonumber\\
&=2\int_{\mathbb R^3}(p-c(t))v\cdot\nabla\phi \phi dx+\int_{\mathbb R^3} |v|^2v\cdot\nabla\phi \phi dx+2\int_{\mathbb R^3}\phi v^i\nabla v^i\cdot\nabla\phi dx\nonumber\\
&\quad+\int_{\mathbb R^3}\nabla_i u^kW_{p^k_j}\nabla_j v^i\phi^2dx+2\int_{\mathbb R^3}\nabla_i u^kW_{p_j^k}v^i\nabla_j\phi\phi dx,\\
&\leq\frac{1}{2}\int_{\R^3}|\D v|^2\phi^2dx+2\int_{\mathbb R^3}(p-c(t))v\cdot\nabla\phi \phi dx\nonumber\\
&\quad+C\int_{\R^3}(|v|^4+|\D u|^4)\phi^2dxds+C\int_{\R^3}|v|^2|\D\phi|^2dx,\nonumber
\end{align}
where $c(t)$ is a temporal function to be chosen later.

Multiplying \eqref{O-F-3} by $\Delta u^i\phi^2$ and integrating over $\mathbb R^3$, we have
\begin{align}\label{3.8}
&\int_{\mathbb R^3}\p_tu^i\Delta u^i\phi^2dx+\int_{\mathbb R^3}(v\cdot\nabla)u^i\Delta u^i\phi^2dx\nonumber\\
&=\int_{\mathbb{R}^3}\nabla_\alpha W_{p_\alpha^i}\Delta u_i\phi^2dx-\int_{\mathbb{R}^3}\nabla_\alpha\left(u^ku^iW_{p_\alpha^k}\right)\Delta u_i\phi^2dx\nonumber\\
&\quad+\int_{\mathbb{R}^3}W_{p^k_\alpha}u^k\nabla_\alpha u^i\Delta u^i\phi^2dx-\int_{\mathbb R^3}W_{u^i}(\Delta u^i-u^iu^k\Delta u^k)\phi^2dx\\
&\quad+\int_{\mathbb{R}^3} W_{p_\alpha^k}\nabla_\alpha u^ku^i\Delta u^i\phi^2dx.\nonumber
\end{align}
Integrating by parts gives
\begin{align}\label{3.9}
\int_{\mathbb R^3}\p_tu^i\Delta u^i\phi^2dx&=-\int_{\mathbb R^3}\nabla u^i_t\nabla u^i\phi^2dx-2\int_{\mathbb R^3}u_t^i\nabla u^i\phi\nabla\phi dx\\
&=-\frac{1}{2}\frac{d}{dt}\int_{\mathbb R^3}|\nabla u|^2\phi^2dx-2\int_{\mathbb R^3}u_t^i\nabla u^i\phi\nabla\phi dx.\nonumber
\end{align}
Integrating by parts twice and using $\D_{\beta}W_{p_\alpha^i}=W_{p_\alpha^ip_\g^j}\D_{\beta\g}^2u^j+W_{p_\alpha^iu^j}\D_{\beta} u^j$ yield
\begin{align}\label{3.10}
&\quad\int_{\mathbb R^3}\nabla_\alpha W_{p^i_\alpha}\Delta u^i\phi^2dx\nonumber\\
&=\int_{\mathbb R^3}\nabla_\beta W_{p^i_\alpha}\nabla_{\beta\alpha}^2u^i\phi^2dx+2\int_{\mathbb R^3}W_{p_\alpha^i}\left(\nabla_{\beta\alpha}^2u^i\nabla_\beta\phi-\Delta u^i\nabla_\alpha\phi\right)\phi dx\nonumber\\
&=\int_{\mathbb R^3}W_{p_\alpha^ip_\gamma^j}\nabla_{\beta\gamma}^2u^j\nabla_{\beta\alpha}^2u^i\phi^2dx+\int_{\mathbb R^3}W_{p_\alpha^iu^j}\nabla_\beta u^j\nabla_{\beta\alpha}^2u^i\phi^2dx\\
&\quad+2\int_{\mathbb R^3}W_{p_\alpha^i}\left(\nabla_{\beta\alpha}^2u^i\nabla_\beta\phi-\Delta u^i\nabla_\alpha\phi\right)\phi dx.\nonumber
\end{align}
Noting that $-u^i\Delta u^i=|\D u|^2$ due to $|u|=1$, one has
\begin{align}\label{3.11}
&\quad-\int_{\mathbb{R}^3}\nabla_\alpha\left(u^ku^iW_{p_\alpha^k}\right)\Delta u_i\phi^2dx+\int_{\mathbb{R}^3}W_{p^k_\alpha}u^k\nabla_\alpha u^i\Delta u^i\phi^2dx\nonumber\\
&=-\int_{\mathbb R^3}\left(u^k\nabla_\alpha W_{p^k_\alpha}+\nabla_\alpha u^kW_{p^k_\alpha}\right)u^i\Delta u^i\phi^2dx\\
&=\int_{\mathbb R^3}\left(u^kW_{p_\alpha^kp_\beta^j}\nabla_{\alpha\beta}^2 u^j+u^kW_{p^k_\alpha u^j}\nabla_\alpha u^j+\nabla_\alpha u^kW_{p^k_\alpha}\right)|\nabla u|^2\phi^2dx.\nonumber
\end{align}
Substituting \eqref{3.9}-\eqref{3.11} into \eqref{3.8} gives
\begin{align}\label{3.12}
&\frac{1}{2}\frac{d}{dt}\int_{\mathbb R^3}|\nabla u|^2\phi^2dx+\int_{\mathbb R^3}W_{p_\alpha^ip_\gamma^j}\nabla_{\beta\gamma}^2u^j\nabla_{\beta\alpha}^2u^i\phi^2dx\nonumber\\
&=-\int_{\mathbb R^3}W_{p_\alpha^iu^j}\nabla_\beta u^j\nabla_{\beta\alpha}^2u^i\phi^2dx-\int_{\mathbb{R}^3} W_{p_\alpha^k}\nabla_\alpha u^ku^i\Delta u^i\phi^2dx\nonumber\\
&\quad-\int_{\mathbb R^3}\left(u^kW_{p_\alpha^kp_\beta^j}\nabla_{\alpha\beta}^2 u^j+u^kW_{p^k_\alpha u^j}\nabla_\alpha u^j+\nabla_\alpha u^kW_{p^k_\alpha}\right)|\nabla u|^2\phi^2dx\\
&\quad+\int_{\mathbb R^3}W_{u^i}(\Delta u^i-u^iu^k\Delta u^k)\phi^2dx+\int_{\mathbb R^3}v\cdot\nabla u^i\Delta u^i\phi^2dx\nonumber\\
&\quad-2\int_{\mathbb R^3}W_{p_\alpha^i}\left(\nabla_{\beta\alpha}^2u^i\nabla_\beta\phi-\Delta u^i\nabla_\alpha\phi\right)\phi dx-2\int_{\mathbb R^3}u_t^i\nabla u^i\phi\nabla\phi dx.\nonumber
\end{align}
Thus, it follows from \eqref{uni-ell}, \eqref{qua-est} and Young's inequality that
\begin{align}\label{2.12}
&\frac{1}{2}\frac{d}{dt}\int_{\mathbb R^3}|\nabla u|^2\phi^2dx+a\int_{\R^3}|\D^2u|^2\phi^2dx\nonumber\\
&\leq\int_{\R^3}(|\D u|^2+|v||\D u|)|\D^2u|\phi^2dx\\
&\quad+\int_{\R^3}(|v||\D u|+|\D^2u|+|\D u|^2)|\D u|\phi|\D\phi|dx\nonumber\\
&\leq\frac{a}{2}\int_{\R^3}|\D^2u|^2\phi^2dx+C\int_{\R^3}(|v|^4+|\D u|^4)\phi^2dx+C\int_{\R^3}|\D u|^2|\D\phi|^2dx,\nonumber
\end{align}
where we have used $|\D u|^2=|u\cdot\Delta u|\leq |\D^2u|$ and $|\p_tu|\leq|v||\D u|+|\D u|^2+|\D^2u|$.
Summing \eqref{2.2} with \eqref{2.12} and integrating over $[0,t]$, we prove \eqref{loc-est1}.
	\end{proof}
	
\begin{lemma}\label{key-lem2}
	Let $(u, v)$ be a smooth solution to \eqref{O-F-1}-\eqref{O-F-3} in $\mathbb R^3\times(0, T)$. Then for any $\phi\in C_0^\infty(\R^3)$, it holds that, for any $t\in(\tau,T)$,
\begin{align}\label{loc-est2}
&\int_{\R^3}(|\D v(t)|^2+|\D^2 u(t)|^2)\phi^6dx+\int_{\tau}^{t}\int_{\R^3}(|\D^2v|^2+a|\D^3 u|^2)\phi^6dxds\nonumber\\
&\leq \int_{\R^3}(|\D v(\tau)|^2+|\D^2 u(\tau)|^2)\phi^6dx\nonumber\\
&\quad+12\int_{\tau}^{t}\int_{\R^3}(p-c(s))(\D v^i\D_i(\phi^5\D\phi)-\Delta v^i\phi^5\D_i\phi)dxds\\
&\quad+C\int_{\tau}^{t}\int_{\R^3}(|v|^2+|\D u|^2)(|\D v|^2+|\D^2 u|^2)\phi^6dxds\nonumber\\
&\quad+C\int_{\tau}^{t}\int_{\R^3} (|\D^2 u|^2+|\D v|^2)|\D\phi|^2\phi^4dxds, \nonumber
\end{align}
where $c(t)$ is an arbitrary temporal function and $C$ is an absolute constant.	
\end{lemma}
\begin{proof}
Multiplying \eqref{O-F-1} by $\Delta v^i\phi^6$ and integrating over $\mathbb R^3$, we have
\begin{align}\label{1}
&\int_{\mathbb R^3}\p_tv^i\Delta v^i\phi^6dx-\int_{\mathbb{R}^3}|\Delta v|^2\phi^6dx=-\int_{\mathbb R^3}v\cdot\nabla v^i\Delta v^i\phi^6dx\\
&\qquad-\int_{\mathbb R^3}\nabla_ip\Delta v^i\phi^6dx-\int_{\mathbb R^3}\nabla_j\left(\nabla_i u^kW_{p_j^k}\right)\Delta v^i\phi^6dx.\nonumber
\end{align}
It follows from integrating by parts, using \eqref{O-F-1} and \eqref{O-F-2} that
\begin{align}\label{2}
\int_{\mathbb R^3}v_t^i\Delta v^i\phi^6dx&=-\frac{1}{2}\frac{d}{dt}\int_{\mathbb R^3}|\nabla v|^2\phi^6dx+6\int_{\mathbb R^3}v_t^i\nabla v^i\phi^5\nabla\phi dx\nonumber\\
&=-\frac{1}{2}\frac{d}{dt}\int_{\mathbb R^3}|\nabla v|^2\phi^6dx-6\int_{\mathbb R^3}v\cdot\nabla v^i\nabla v^i\phi^5\nabla \phi dx\nonumber\\
&\quad+6\int_{\mathbb R^3}(p-c)\nabla v^i\nabla_i(\phi^5\nabla \phi) dx+6\int_{\mathbb R^3}\Delta v^i\nabla v^i\phi^5\nabla \phi dx\\
&\quad-6\int_{\mathbb R^3}\nabla_j(\nabla_iu^kW_{p_j^k})\nabla v^i\phi^5\nabla \phi dx,\nonumber
\end{align}
\begin{equation}\label{3}
-\int_{\mathbb{R}^3}|\Delta v|^2\phi^6dx=-\int_{\mathbb R^3}|\nabla^2v|\phi^6dx-6\int_{\mathbb R^3}\nabla_iv(\nabla_{ij}^2v\nabla_j\phi-\Delta v\nabla_i\phi)\phi^5dx
\end{equation}
and
\begin{equation}\label{4}
\int_{\mathbb R^3}\nabla_ip\Delta v^i\phi^6dx=-6\int_{\mathbb R^3}(p-c)\Delta v^i\phi^5\nabla_i\phi dx,
\end{equation}
where $c(t)\in \R$ is a temporal function to be chosen later. Substituting \eqref{2}-\eqref{4} into \eqref{1}, and using \eqref{qua-est} and Young's inequality, we obtain
\begin{align}\label{2.14}
&\frac{1}{2}\frac{d}{dt}\int_{\mathbb R^3}|\nabla v|^2\phi^6dx+\int_{\mathbb R^3}|\nabla^2v|^2\phi^6dx\nonumber\\
&=6\int_{\mathbb R^3}(p-c)\left(\nabla v^i\nabla_i(\phi^5\nabla \phi)-\Delta v^i\phi^5\nabla_i\phi\right) dx\nonumber\\
&\quad+\int_{\mathbb R^3}\left(v\cdot\D v^i+\nabla_j(\nabla_iu^kW_{p_j^k})\right)\left(\Delta v^i\phi^6-6\nabla v^i\phi^5\nabla \phi\right) dx\\
&\quad+6\int_{\mathbb R^3}\Delta v^i\nabla v^i\phi^5\nabla \phi-6\int_{\mathbb R^3}\nabla_iv(\nabla_{ij}^2v\nabla_j\phi-\Delta v\nabla_i\phi)\phi^5dx\nonumber\\
&\leq\frac{1}{2}\int_{\R^3}|\D^2v|^2\phi^6dx+6\int_{\mathbb R^3}(p-c)\left(\nabla v^i\nabla_i(\phi^5\nabla \phi)-\Delta v^i\phi^5\nabla_i\phi\right) dx\nonumber\\
&\quad+C\int_{\R^3}(|v|^2|\D v|^2+|\D u|^2|\D^2u|^2)\phi^6dx+C\int_{\R^3}|\D v|^2|\D\phi|^2\phi^4dx.\nonumber
\end{align}
Differentiating \eqref{O-F-3} in $x_\beta$, multiplying the resulting equation by $-\nabla_\beta\Delta u^i\phi^6$ and integrating over $\R^3$ give
\begin{align}\label{2.15}
&-\int_{\mathbb R^3}\partial_t\nabla_\beta u^i\nabla_\beta\Delta u^i\phi^6dx+\int_{\mathbb{R}^3}\nabla_{\beta\alpha}^2W_{p_\alpha^i}\nabla_\beta\Delta u^i\phi^6dx\nonumber\\
&=\int_{\mathbb R^3}\nabla_\beta\left(\nabla_\alpha(u^ku^iW_{p^k_\alpha})-W_{p^k_\alpha}u^k\nabla_\alpha u^i\right)\nabla_\beta\Delta u^i\phi^6dx\nonumber\\
&\quad+\int_{\mathbb R^3}\nabla_\beta(v\cdot\nabla u^i)\nabla_\beta\Delta u^i\phi^6dx+\int_{\mathbb{R}^3}\nabla_\beta W_{u^i}\nabla_\beta\Delta u^i\phi^6dx\\
&\quad-\int_{\mathbb R^3}\nabla_\beta\left(W_{u^k}u^ku^i+W_{p^k_\alpha}\nabla_\alpha u^ku^i\right)\nabla_\beta\Delta u^i\phi^6dx.\nonumber
\end{align}
Integrating by parts yields that
\begin{align}\label{2.16}
-\int_{\mathbb R^3}\partial_t\nabla_\beta u^i\nabla_\beta\Delta u^i\phi^6dx=\frac{1}{2}\frac{d}{dt}\int_{\mathbb R^3}|\nabla^2 u|^2\phi^6dx+6\int_{\mathbb R^3}\partial_t\nabla_\beta u^i\nabla_{\beta\alpha}^2u^i\phi^5\nabla_\alpha \phi dx.
\end{align}
It follows from Young's inequality that
\begin{align}\label{2.17}
&6\int_{\mathbb R^3}\partial_t\nabla_\beta u^i\nabla_{\beta\alpha}^2u^i\phi^3\nabla_\alpha\phi dx\nonumber\\
&\leq \frac{a}{8}\int_{\R^3}|\D^3u|^2\phi^6dx+C\int_{\R^3}(|\D v|^2+|\D^2u|^2)|\D\phi|^2\phi^4dx\\
&\quad+C\int_{\R^3}(|v|^2+|\D u|^2)|\D^2u|^2\phi^6dx,\nonumber
\end{align}
where we have used
$$|\p_t\D u|\leq C(|\D v||\D u|+|v||\D^2u|+|\D u||\D^2u|+|\D^3 u|),$$
which follows from \eqref{O-F-3} and \eqref{qua-est}.
Integrating by parts twice, we note
\begin{align}\label{2.18}
\int_{\mathbb{R}^3}\nabla_{\beta\alpha}^2W_{p_\alpha^i}\nabla_\beta\Delta u^i\phi^4dx=&6\int_{\mathbb R^3}\nabla_\beta W_{p_\alpha^i}(\nabla_{\alpha\beta\gamma}^3u^i\nabla_\gamma\phi-\nabla_\beta\Delta u^i\D_{\alpha}\phi)\phi^5dx\\
&+\int_{\mathbb{R}^3}\nabla_{\gamma\beta}^2W_{p_\alpha^i}\nabla_{\alpha\beta\gamma}^3 u^i\phi^6dx.\nonumber
\end{align}
Since $W(u,p)$ is quadratic in $u$ and $p$, we have
\begin{align*}
\nabla^2_{\g \b}W_{p_{\a}^i}=& \nabla_{\g}\left(
W_{u^jp_{\a}^i}(u,\nabla u) \nabla_{\b}
u^j+W_{p_{\a}^i}(u,\nabla \nabla_{\b}u)\right)\\
=&W_{u^jp_{\a}^i}(u,\nabla u) \nabla^2_{\g\b}
u^j+W_{u^jp_{\a}^i}(\nabla_\gamma u,\nabla u)\nabla_{\b}
u^j\nonumber\\
&+W_{u^jp_{\a}^i}(u,\nabla_\gamma\nabla u)\nabla_{\b}
u^j+W_{u^jp_{\a}^i}(u,\nabla \nabla_{\b}u) \nabla_{\g} u^j\\
&+W_{p_l^jp_{\a}^i}(u,\nabla \nabla_{\b}u) \nabla^3_{\b \g l}
u^j.
 \end{align*}
It follows from \eqref{uni-ell} that
\begin{align}\label{2.19}
\int_{\mathbb R^3}W_{p_\alpha^ip_l^j}(u,\nabla_\beta\nabla u)\nabla_{\beta\gamma l}^3u^j\nabla_{\alpha\beta\gamma}^3u^i\phi^6dx
\geq a\int_{\mathbb R^3}|\nabla^3u|^2\phi^6dx.
\end{align}
By using \eqref{qua-est} and Young's inequality, one has
\begin{align}\label{2.20}
&\int_{\mathbb R^3}[W_{u^jp_\alpha^i}(u,\nabla u)\nabla_{\gamma\beta}u^j+W_{u^jp_\alpha^i}(u,\nabla_\gamma\nabla u)\nabla_\beta u^j\nonumber\\
&\qquad+W_{u^jp_{\a}^i}(\nabla_\gamma u,\nabla u)\nabla_{\b}
u^j+W_{u^jp_\alpha^i}(u,\nabla_\beta\nabla u)\nabla_\gamma u^j]\nabla_{\alpha\beta\gamma}^3u^i\phi^6dx\\
&\leq\frac{a}{16}\int_{\mathbb R^3}|\nabla^3u|^2\phi^6dx+C\int_{\mathbb R^3}|\nabla u|^2|\nabla^2u|^2\phi^6dx,\nonumber
\end{align}
where we have used the fact $|\nabla u|^4=|\Delta u|^2$ for $|u|=1$. It is clear that
\begin{align}\label{2.21}
&6\int_{\mathbb R^3}\nabla_\beta W_{p_\alpha^i}(\nabla_{\alpha\beta\gamma}^3u^i\nabla_\gamma\phi-\nabla_\beta\Delta u^i\D_{\alpha}\phi)\phi^5dx\nonumber\\
&\leq C\int_{\mathbb R^3}(|\nabla u|^2+|\nabla^2 u|)|\nabla^3u||\nabla\phi|\phi^5dx\\
&\leq\frac{a}{16}\int_{\mathbb R^3}|\nabla^3u|^2\phi^6dx+C\int_{\mathbb R^3}(\nabla^2u|^2+|\nabla u|^4)|\nabla\phi|^2\phi^4dx.\nonumber
\end{align}
Substituting \eqref{2.19}-\eqref{2.21} into \eqref{2.18}, we have
\begin{align}\label{2.22}
&\int_{\mathbb{R}^3}\nabla_{\beta\alpha}^2W_{p_\alpha^i}\nabla_\beta\Delta u^i\phi^6dx\\
&\geq\frac{7a}{8}\int_{\R^3}|\D^3u|^2\phi^6dx-C\int_{\R^3}|\D u|^2|\D^2u|^2\phi^6dx-C\int_{\R^3}|\D ^2u|^2|\D\phi|^2\phi^4dx.\nonumber
\end{align}
For the terms on the right hand of \eqref{2.15}, since $|u|^2=1$ implies
$$u^i\nabla_\beta\Delta u^i+\nabla_\beta u^i\Delta u^i=-\nabla_\beta|\nabla u|^2,$$
one has
\begin{align}\label{2.23}
&\int_{\mathbb R^3}\nabla_\beta\left(\nabla_\alpha(u^ku^iW_{p^k_\alpha})+W_{p^k_\alpha}u^k\nabla u^i\right)\nabla_\beta\Delta u^i\phi^6dx\nonumber\\
&=\int_{\mathbb R^3}\nabla_{\alpha\beta}^2\left(u^kW_{p_\alpha^k}\right)u^i\nabla_\beta\Delta u^i\phi^6dx+\int_{\mathbb R^3}\nabla_\beta u^i\nabla_{\alpha}\left(u^kW_{p_\alpha^k}\right)\nabla_\beta\Delta u^i\phi^6dx\\
&\leq\frac{a}{8}\int_{\mathbb R^3}|\nabla^3u|^2\phi^6dx+C\int_{\mathbb R^3}|\nabla u|^2|\nabla^2u|^2\phi^6dx.\nonumber
\end{align}
From \eqref{qua-est}, other remaining terms on the right hand of \eqref{2.15} can be easily controlled by
\begin{align}\label{2.24}
\frac{a}{8}\int_{\mathbb R^3}|\nabla^3u|^2\phi^6dx+C\int_{\mathbb R^3}(|\nabla^2u|^2|\nabla u|^2+|v|^2|\nabla^2u|^2+|\nabla v|^2|\nabla u|^2)\phi^6dx.
\end{align}
Substituting \eqref{2.16}, \eqref{2.17} and \eqref{2.22}-\eqref{2.24} into \eqref{2.15}, we have
\begin{align}\label{2.25}
&\frac{1}{2}\frac{d}{dt}\int_{\R^3}|\D^2u|^2\phi^6dx+\frac{a}{2}\int_{\R^3}|\D^3u|^2\phi^6dx\\
&\leq C\int_{\R^3}(|v|^2+|\D u|^2)|\D^2u|^2\phi^6dx+C\int_{\R^3}|\D^2 u||\D\phi|^2\phi^4dx.\nonumber
\end{align}
Summing \eqref{2.14} with \eqref{2.25} and integrating over $[\tau,t]$ give (\ref{loc-est2}).
\end{proof}

The following lemma gives local estimates of pressure under the smallness assumption of $L^3$ norm of $(v,\D u)$. The idea of proof, which has been used in \cite{LR,HW}, relies on the well-known Calderon-Zygmund  theory \cite{Stein} and covering arguments.
\begin{lemma}\label{key-lemma3}
	Let $(u, v)$ be a smooth solution to \eqref{O-F-1}-\eqref{O-F-3} in $\mathbb R^3\times(0, T)$ and $\phi$ be a cut-off function satisfying $0\leq \phi\leq 1$, $\text{supp } \phi\subset B_R(x_0)$ for some $x_0\in \R^3$ and $|\D \phi|\leq\frac{C}{R}$. Assume that
	\begin{equation}
	\sup_{0\leq s\leq t,y\in\R^3}\int_{B_R(y)}|\D u(x,s)|^3+|v(x,s)|^3dx\leq \varepsilon_1^3.
	\end{equation}
	Then, there exists a $c(t)\in \R$ such that the pressure $p$ satisfies the follow estimates
	\begin{align}
	&\int_{0}^{t}\int_{\R^3}(p-c(s))^\frac{3}{2}\phi^\frac{3}{2}dxds\leq C\varepsilon_1^3t,\label{pres1}\\
	&\int_{0}^{t}\int_{\R^3}(p-c(s))^2\phi^2dxds\leq C\varepsilon_1^2\int_{0}^{t}\int_{\R^3}(|\D^2u|^2+|\D v|^2)\phi^2dxds\label{pres2}\\
	&\qquad\qquad\qquad\qquad\qquad\qquad+\frac{C\varepsilon_1^2t}{R^2}\Big(\sup_{0\leq s\leq t,y\in\R^3}\int_{B_{\frac{R}{2}}(y)}|\D u|^2+|v|^2dx\Big),\nonumber\\
	&\int_{\tau}^{t}\int_{\R^3}(p-c(s))^3\phi^6dxds\leq C\int_{\tau}^{t}\int_{\R^3}(|\D u|^6+|v|^6)\phi^6dx+\frac{C(t-\tau)}{R^3}\varepsilon_1^6.\label{pres3}
	\end{align}
\end{lemma}
\begin{proof}
Taking divergence in both sides of \eqref{O-F-1},  the pressure $p$ satisfies the elliptic equation
\begin{equation*}
-\Delta p=\nabla_{ij}[\nabla_iu^kW_{p_j^k}+v^jv^i],\quad\text{on}~~\mathbb{R}^3\times[0,T],
\end{equation*}
which implies
$$p=\mathcal{R}_i\mathcal{R}_j(F^{ij}),\quad F^{ij}:=\nabla_iu^kW_{p_j^k}+v^jv^i,$$
where $\mathcal{R}_i$ is the $i$-th Riesz transform on $\mathbb R^3$. Then, we have
\begin{equation}\label{3.21}
(p-c)\phi=\mathcal{R}_i\mathcal{R}_j(F^{ij}\phi)+[\phi,\mathcal{R}_i\mathcal{R}_j](F^{ij})-c\phi
\end{equation}for a cut-off function $\phi$,
where the commutator $[\phi,\mathcal{R}_i\mathcal{R}_j]$ is defined by
$$[\phi,\mathcal{R}_i\mathcal{R}_j](\cdot)=\phi\mathcal{R}_i\mathcal{R}_j(\cdot)-\mathcal{R}_i\mathcal{R}_j(\cdot\,\phi).$$
Since $|F^{ij}|\leq C(|\D u|^2+|v|^2)$ and the Riesz operator maps $L^q$ into $L^q$ spaces for any $1<q<+\infty$, we have
\begin{align}\label{pre-1-1}
&\int_{0}^{t}\int_{\R^3}|\mathcal{R}_i\mathcal{R}_j(F^{ij}\phi)|^\frac{3}{2}dxds\leq \int_{0}^{t}\int_{\R^3}(|\D u|^3+|v|^3)\phi^\frac{3}{2}dxds\leq \varepsilon_1^3t,
\end{align}
\begin{align}\label{pre-1-2}
&\int_{0}^{t}\int_{\R^3}|\mathcal{R}_i\mathcal{R}_j(F^{ij}\phi)|^2dxds\leq  \int_{0}^{t}\int_{\R^3}(|\D u|^4+|v|^4)\phi^2dxds\\
&\leq\int_{0}^{t}\left(\int_{B_R(x_0)}|\D u|^3+|v|^3dx\right)^\frac{2}{3}\left(\int_{\R^3}|\D u\phi|^6+|v\phi|^6dx\right)^\frac{1}{3}ds\nonumber\\
&\leq C\varepsilon_1^2\left(\int_{0}^{t}\int_{\R^3}(|\D^2u|^2+|\D v|^2)\phi^2dxds+\int_{0}^{t}\int_{\R^3}(|\D u|^2+| v|^2)|\D\phi|^2dxds\right)\nonumber\\
&\leq C\varepsilon_1^2\int_{0}^{t}\int_{\R^3}(|\D^2u|^2+|\D v|^2)\phi^2dxds+\frac{C\varepsilon_1^2t}{R^2}\sup_{0\leq s\leq t,y\in\R^3}\int_{B_\frac{R}{2}(y)}|\D u|^2+|v|^2dx\nonumber
\end{align}
and
\begin{align}\label{pre-1-3}
\int_{\tau}^{t}\int_{\R^3}|\mathcal{R}_i\mathcal{R}_j(F^{ij}\phi^2)|^3dxds\leq  \int_{\tau}^{t}\int_{\R^3}(|\D u|^6+|v|^6)\phi^6dxds,
\end{align}
where we have used the covering of ball $B_R(x_0)$ by a fixed number of balls with radius $\frac{R}{2}$ in the last step of \eqref{pre-1-2}.

Now, we
estimate  the commutator. Since $\text{supp}\, \phi\subset B_R(x_0)$, the commutator can be expressed as
\begin{align}\label{com}
&[\phi,\mathcal{R}_i\mathcal{R}_j](F^{ij})(x,t)-c(t)\phi(x)\nonumber\\
&=\int_{\mathbb R^3}\frac{(\phi(x)-\phi(y))(x_i-y_i)(x_j-y_j)}{|x-y|^5}F^{ij}(y,t)dy-c(t)\phi(x)\nonumber\\
&=\int_{B_{2R}(x_0)}\frac{(\phi(x)-\phi(y))(x_i-y_i)(x_j-y_j)}{|x-y|^5}F^{ij}(y,t)dy\\
&\quad+\phi(x)\Big[\int_{\mathbb R^3\backslash B_{2R}(x_0)}\frac{(x_i-y_i)(x_j-y_j)}{|x-y|^5}F^{ij}(y,t)dy-c(t)\Big]\nonumber\\
&=:f_1(x,t)+f_2(x,t).\nonumber
\end{align}
Note that
\begin{align*}
|f_1(x,t)|\leq CR^{-1}\int_{\mathbb R^3}\frac{(|\nabla u|^2+|v|^2)\chi_{B_{2R(x_0)}}}{|x-y|^2}dy,
\end{align*}
and   the Hardy-Littlewood-Sobolev inequality hold by
$$\|I_\alpha(f)\|_{L^q(\R^n)}\leq C\|f\|_{L^r(\R^n)},\quad\frac{1}{q}=\frac{1}{r}-\frac{\alpha}{n}, $$
where $I_\alpha(f):=\int_{\R^n}\frac{f(y)}{|x-y|^{n-\alpha}}dy$.
Then it follows from H\"{o}lder's inequality and a standard covering argument that
\begin{align}\label{prec1-1}
&\int_{0}^{t}\int_{\R^3}|f_1(x,s)|^\frac{3}{2}dxds\nonumber\\
&\leq CR^{-\frac{3}{2}}\int_{0}^{t}\|(|\D u|^2+|v|^2)\chi_{B_{2R}(x_0)}\|_{L^1(\R^3)}^\frac{3}{2}ds\nonumber\\
&\leq C\int_{0}^{t}\int_{B_{2R(x_0)}}|\D u|^3+|v|^3dxds\\
&\leq Ct\sup_{0\leq s\leq t,y\in\R^3}\int_{B_{\frac{R}{2}}(y)}|\D u|^3+|v|^3dx\leq C\varepsilon_1^3t.\nonumber
\end{align}
Similarly, we have
\begin{align}\label{prec1-2}
&\int_{0}^{t}\int_{\R^3}|f_1(x,s)|^2dxds\leq CR^{-2}\int_{0}^{t}\|(|\nabla u|^2+|v|^2)\chi_{B_{2R(x_0)}}\|_{L^\frac{6}{5}(\R^3)}^2ds\\
&\leq CR^{-2}\int_{0}^{t}\|(|\D u|+|v|)\chi_{B_{2r}(x_0)}\|_{L^3(\R^3)}^2\|(|\D u|+|v|)\chi_{B_{2r}(x_0)}\|_{L^2(\R^3)}^2ds\nonumber\\
&\leq \frac{C\varepsilon_1^2}{R^2}\int_{0}^{t}\int_{B_{2R}(x_0)}(|\D u|^2+|v|^2)dx\leq \frac{C\varepsilon_1^2t}{R^2}\sup_{0\leq s\leq t,y\in\R^3}\int_{B_{\frac{R}{2}}(y)}(|\D u|^2+|v|^2)dx\nonumber
\end{align}
and
\begin{align}\label{prec1-3}
&\int_{\tau}^{t}\int_{\R^3}|f_1(x,s)|^3dxds\leq CR^{-3}\int_{\tau}^{t}\|(|\D u|^2+|v|^2)\chi_{B_{2R}(y)}\|_{L^\frac{3}{2}(\R^3)}^3ds\nonumber\\
&\leq CR^{-3}\int_{\tau}^{t}\left(\int_{B_{2R}(y)}|\D u|^3+|v|^3dx\right)^2ds\\
&\leq \frac{C(t-\tau)}{R^3}(\sup_{\tau\leq s\leq t,x_0\in\R^3}\int_{B_{\frac{R}{2}}(x_0)}|\D u|^3+|v|^3dx)^2\leq \frac{C\varepsilon_1^6(t-\tau)}{R^3}.\nonumber
\end{align}
As in Lemma 3.2 of \cite{HW}, to estimate the term involving $f_2(x,t)$, we choose
$$c(t)=\int_{\mathbb R^3\backslash B_{2R}(x_0)}\frac{(x_{0i}-y_i)(x_{0j}-y_j)}{|x_0-y|^5}F^{ij}(y,t)dy,$$
which is finite for any approximation data $(v,u)\in L^2(\R^3)\times \dot{H}^1(\R^3)$. Then, one has
\begin{align*}
|f_2(x,t)|&\leq\Big|\phi(x)\int_{\mathbb R^3\backslash B_{2R}(x_0)}\big(\frac{(x_{i}-y_i)(x_{j}-y_j)}{|x-y|^5}-\frac{(x_{0i}-y_i)(x_{0j}-y_j)}{|x_0-y|^5}\big)F^{ij}(y,t)dy\Big|\nonumber\\
&\leq CR\phi(x)\int_{\mathbb R^3\backslash B_{2R}(x_0)}\frac{(|\nabla u|^2+|v|^2)(y)}{|x_0-y|^4}dy
\end{align*}
due to the fact (c.f. \cite{Stein}) that
\begin{align*}
|\frac{(x_{i}-y_i)(x_{j}-y_j)}{|x_-y|^5}-\frac{(x_{0i}-y_i)(x_{0j}-y_j)}{|x_0-y|^5}|\leq C \frac{|x_0-x|}{|x_0-y|^4}.
\end{align*}
Using H\"{o}lder's inequality and a standard covering argument, we obtain
\begin{align*}
|f_{2}(x,t)|\leq CR^{-3}\phi(x)\sup_{y\in\mathbb R^3}\int_{B_R(y)}(|\nabla u|^2+|v|^2)dx
\end{align*}
which implies
\begin{align}\label{prec2-1}
&\int_{0}^{t}\int_{\R^3}|f_{2}(x,t)|^\frac{3}{2}dxds\leq C\int_{0}^{t}\int_{\R^3}\left(R^{-3}\phi(x)\sup_{y\in\mathbb R^3}\int_{B_R(y)}(|\nabla u|^2+|v|^2)\right)^\frac{3}{2}dxds\\
&\leq CR^{-3}\int_{0}^{t}\int_{\R^3}\phi^\frac{3}{2}\left(\sup_{y\in\R^3}(\int_{B_R(y)}|\D u|^3+|v|^3)^\frac{2}{3}\right)^\frac{3}{2}dxds\leq C\varepsilon_1^3t,\nonumber
\end{align}
\begin{align}\label{prec2-2}
&\int_{0}^{t}\int_{\R^3}|f_{2}(x,t)|^2dxds\leq C\int_{0}^{t}\int_{\R^3}\left(R^{-3}\phi(x)\sup_{y\in\mathbb R^3}\int_{B_R(y)}(|\nabla u|^2+|v|^2)\right)^2dxds\\
&\leq CR^{-5}\varepsilon_1^2\int_{0}^{t}\int_{\R^3}\phi^2\left(\sup_{y\in\mathbb R^3}\int_{B_{\frac{R}{2}}(y)}(|\nabla u|^2+|v|^2)\right)dxds\nonumber\\
&\leq \frac{Ct}{R^2}\varepsilon_1^2\sup_{0\leq s\leq t,y\in R^3}\int_{B_{\frac{R}{2}}(y)}(|\nabla u|^2+|v|^2)dx\nonumber
\end{align}
and
\begin{align}\label{prec2-3}
&\int_{\tau}^{t}\int_{\R^3}|f_{2}(x,t)|^3dxds\leq C\int_{\tau}^{t}\int_{\R^3}\left(R^{-3}\phi(x)\sup_{y\in\mathbb R^3}\int_{B_R(y)}(|\nabla u|^2+|v|^2)\right)^3dxds\\
&\leq C\int_{0}^{t}\int_{\R^3}\left(R^{-2}\phi(x)\sup_{y\in\mathbb R^3}(\int_{B_R(y)}(|\nabla u|^3+|v|^3))^\frac{2}{3}\right)^3dxds\leq \frac{C\varepsilon_1^6(t-\tau)}{R^3}.\nonumber
\end{align}
Collecting \eqref{3.21}, \eqref{pre-1-1}, \eqref{com}, \eqref{prec1-1} and \eqref{prec2-1} yields \eqref{pres1}. Next, combining \eqref{3.21}, \eqref{com}, \eqref{pre-1-2}, \eqref{prec1-2} with  \eqref{prec2-2} gives \eqref{pres2}. Finally, \eqref{pres3} follows from \eqref{3.21}, \eqref{pre-1-3}, \eqref{com}, \eqref{prec1-3} and \eqref{prec2-3}. Therefore, the desired results are obtained.
\end{proof}

Based on the above local estimates lemmas, we have the following propositions.
\begin{prop}\label{prop1}
	Let $(u, v)$ be a smooth solution to \eqref{O-F-1}-\eqref{O-F-3} in $\mathbb R^3\times(0, T)$. Then, there are constants $\varepsilon_1>0$ and $\sigma>0$ such that, if
	\begin{equation}\label{2.39}
	\esssup_{0\leq s\leq \sigma R^2,y\in\mathbb R^3}\int_{B_{R_0}(y)}|\nabla u(x,s)|^3+|v(x,s)|^3dx< \varepsilon_1^3
	\end{equation}
	for some $R_0,R>0$, then for any $x_0\in\R^3$, $r\leq R_0$ and $t\leq \min\{\sigma r^2,\sigma R^2\}$, we have
		\begin{align}\label{p-1}
		&\frac{1}{r}\int_{\R^3}(|v(t)|^2+|\D u(t)|^2)\phi^2 dx+\frac{1}{r}\int_{0}^{t}\int_{\R^3}(|\D v|^2+|\D^2u|^2)\phi^2dxds\nonumber\\
		&\leq \frac{1}{r}\int_{\R^3}(|v_0|^2+|\D u_0|^2)\phi^2dx+\frac{C}{r}\int_{0}^{t}\int_{\R^3}(|v|^2+|\D u|^2)|\D\phi|^2dxds\\
		&\qquad+\frac{C}{r}\int_{0}^{t}\int_{\R^3}(p-c(s))v\cdot\D\phi\phi dxds,\nonumber
		\end{align}
		where $\phi$ is the cut-off function compactly supported in $B_r(x_0)$ with $\phi\equiv 1$ on $B_\frac{r}{2}(x_0)$ and $|\D\phi|\leq Cr^{-1}$. Moreover, it holds for any $x_0\in\R^3$, $r\leq R_0$ and $t\leq \min\{\sigma r^2,\sigma R^2\}$ that
		\begin{align}\label{p-1-1}
		&\frac{1}{r}\int_{B_{\frac{r}{2}}(x_0)}|v(t)|^2+|\D u(t)|^2dx+\frac{1}{r}\int_{0}^{t}\int_{B_{\frac{r}{2}}(x_0)}|\D v|^2+|\D^2u|^2dxds\\
		&\leq\sup_{y\in\R^3}\frac{C}{r}\int_{B_r(y)}|v_0|^2+|\D u_0|^2dx.\nonumber
		\end{align}
\end{prop}
\begin{proof}
For any $x_0\in \R^3$ and $r\leq R_0$, let $\phi$  in \eqref{loc-est1} be a cut-off function compactly supported in $B_r(x_0)$ with $\phi\equiv 1$ on $B_\frac{r}{2}(x_0)$ and $|\D\phi|\leq Cr^{-1}$. It follows from H\"{o}lder's inequality, the Sobolev embedding theorem and \eqref{2.39} that
\begin{align}\label{2.40}
&C\int_{0}^{t}\int_{\mathbb R^3}(|\nabla u|^4+|v|^4)\phi^2dx\nonumber\\
&\leq C\int_{0}^{t}\left(\int_{B_r(x_0)}|\nabla u|^3+|v|^3\right)^\frac{2}{3}\left(\int_{\mathbb R^3}(|\phi\nabla u|^6+|\phi v|^6)dx\right)^\frac{1}{3}ds\\
&\leq C\varepsilon_1^2\left(\int_{0}^{t}\int_{\mathbb{R}^3}(|\nabla^2 u|^2+|\nabla v|^2)\phi^2dx+\int_{\R^3}(|\nabla u|^2+|v|^2)|\D\phi|^2dx\right).\nonumber
\end{align}
Choosing $\varepsilon_1$ small enough such that $C\varepsilon_1<\min\{\frac{1}{4},\frac{a}{4}\}$ and substituting \eqref{2.40} into \eqref{loc-est1}, we complete the proof of \eqref{p-1}. To prove \eqref{p-1-1}, we use covering arguments to estimate terms on the right hand side of \eqref{p-1}. It follows from Young's equality, \eqref{pres2} and covering argument that
\begin{align}\label{2.41}
&C\int_{0}^{t}\int_{\R^3}(p-c(s))v\cdot\D\phi\phi dxds\\
&\leq\int_{0}^{t}\int_{\R^3}(p-c(s))^2\phi^2dxds+C\int_{0}^{t}\int_{\R^3}|v|^2|\D\phi|^2dxds\nonumber\\
&\leq C\varepsilon_1^2\int_{0}^{t}\int_{\R^3}(|\D^2u|^2+|\D v|^2)\phi^2dxds+\frac{C\varepsilon_1^2t}{r^2}\left(\sup_{0\leq s\leq t,y\in\R^3}\int_{B_{\frac{r}{2}}(y)}|\D u|^2+|v|^2dx\right)\nonumber
\end{align}
and
\begin{align}\label{2.42}
C\int_{0}^{t}\int_{\R^3}(|\nabla u|^2+|v|^2)|\D\phi|^2dxds\leq \frac{Ct}{r^2}\sup_{0\leq s\leq t,y\in\R^3}\int_{B_{\frac{r}{2}}(y)}|\D u|^2+|v|^2dx.
\end{align}
Substituting \eqref{2.40}-\eqref{2.42} into \eqref{loc-est1}, dividing the resulting equation by $r$, choosing $\sigma$ sufficiently small such that $C\sigma<\frac{1}{2}$ and taking super-mum with respect to $x_0$ in $\R^3$ and $0\leq t\leq\min\{\sigma R^2,\sigma r^2\}$, we have  \eqref{p-1-1}.
		\end{proof}
\begin{prop}\label{prop2}
	Let $(u, v)$ be a smooth solution to \eqref{O-F-1}-\eqref{O-F-3} in $\mathbb R^3\times(0, T)$. Then, there are constants $\varepsilon_1>0$ and $\sigma>0$ such that, if
	\begin{equation}\label{2.39}
	\esssup_{0\leq s\leq \sigma R^2,y\in\mathbb R^3}\int_{B_{R_0}(y)}|\nabla u(x,s)|^3+|v(x,s)|^3dx< \varepsilon_1^3
	\end{equation}
	for some $R_0,R>0$, then for any $x_0\in\R^3$, $R\leq R_0$ and $0<\delta t<\tau<t=\sigma R^2$, we have
		\begin{align}\label{p-2}
		&R\int_{\R^3}(|\D v(\tau)|^2+|\D^2 u(\tau)|^2)\phi^6dx\nonumber\\
		&\leq \frac{C}{R}\int_{B_R(x_0)}|\D u_0|^2+|v_0|^2dx+\frac{C}{R}\int_{0}^{t}\int_{\R^3}(|v|^2+|\D u|^2)|\D\phi|^2\phi^4dxds\\
		&\quad+\frac{C}{R}\int_{0}^{t}\int_{\R^3}(p-c(s))v\cdot\D\phi\phi^5 dxds+\frac{C}{R}\int_{\delta t}^{\tau}\int_{\R^3}(p-c(s))^2\phi^4dxds,\nonumber
		\end{align}
		where $\delta\in(\frac{1}{2},1)$ is a constant and $\phi$ is the cut-off function compactly supported in $B_\frac{R}{2}(x_0)$ with $\phi\equiv 1$ on $B_\frac{R}{4}(x_0)$, $|\D\phi|\leq CR^{-1}$ and $|\D^2\phi|\leq CR^{-2}$. Moreover, it holds for any $x_0\in\R^3$, $R\leq R_0$ and $0<\delta t<\tau<t=\sigma R^2$ that
		\begin{equation}\label{p-2-2}
		R\int_{B_{\frac{R}{4}}(x_0)}(|\D v(\tau)|^2+|\D^2 u(\tau)|^2)dx\leq \sup_{y\in\R^3}\frac{C}{R}\int_{B_R(y)}|v_0|^2+|\D u_0|^2dx.
		\end{equation}
\end{prop}	
\begin{proof}
	For any $x_0\in\R^3$, let $\phi$ in \eqref{loc-est2} be a cut-off function  compactly supported in $B_{\frac{R}{2}}(x_0)$ with $\phi\equiv 1$ on $B_\frac{R}{4}(x_0)$, $|\D\phi|\leq CR^{-1}$ and $|\D^2\phi|\leq CR^{-2}$. By the mean value theorem, \eqref{p-1} and \eqref{p-1-1}, for $t=\sigma R^2$, $r=R$, there exists a $\delta\in(\frac{1}{2},1)$ such that
	\begin{align}\label{2.44}
	&\frac{t}{2R}\int_{\R^3}(|\D v(\delta t)|^2+|\D^2 u(\delta t)|^2)\phi^6dx\nonumber\\
	&\leq \frac{1}{R}\int_{\R^3}(|v_0|^2+|\D u_0|^2)\phi^6dx+\frac{C}{R}\int_{0}^{t}\int_{\R^3}(|v|^2+|\D u|^2)|\D\phi|^2\phi^4dxds\\
	&\qquad+\frac{C}{R}\int_{0}^{t}\int_{\R^3}(p-c(s))v\cdot\D\phi\phi^5 dxds.\nonumber
	\end{align}
	By using Young's inequality, one has
	\begin{align}\label{2.44-1}
	&12\int_{\delta t}^{\tau}\int_{\R^3}(p-c(t))(\D v^i\D_i(\phi^5\D\phi)-\Delta v^i\phi^5\D_i\phi)dxds\nonumber\\
	&\leq\frac{1}{2}\int_{\delta t}^{\tau}\int_{\R^3}|\D^2 v|^2\phi^6dxds+\frac {C}{R^2}\int_{\delta t}^{\tau}\int_{\R^3} |\D v|^2\phi^4 dxds\\
	&\quad+\frac{C}{R^2}\int_{\delta t}^{\tau}\int_{\R^3}(p-c(t))^2\phi^4dxds.\nonumber
	\end{align}
	It follows from H\"{o}lder's inequality and the Sobolev embedding theorem that
	\begin{align}\label{2.45}
	&\int_{\delta t}^{\tau}\int_{\mathbb R^3}(|v|^2+|\nabla u|^2)(|\nabla v|^2+|\nabla^2u|^2)\phi^6dxds\nonumber\\
	&\leq\int_{\delta t}^{\tau}\Big(\int_{B_R(x_0)}|\nabla u|^3+|v|^3dx\Big)^\frac{2}{3}\Big(\int_{\mathbb R^3}(|\nabla^2u|^6+|\nabla v|^6)\phi^{18}\Big)^\frac{1}{3}ds\nonumber\\
	&\leq C\varepsilon_1^2\int_{\delta t}^{\tau}\int_{\mathbb R^3}(|\nabla^3u|^2+|\nabla^2v|^2)\phi^6dxds \\ &\quad+C\varepsilon_1^2\int_{\delta t}^{\tau}\int_{\R^3}(|\nabla^2u|+|\nabla v|^2)|\D\phi|^2\phi^4dxds.\nonumber
	\end{align}
	Substituting \eqref{2.44}-\eqref{2.45} into \eqref{loc-est2} and multiplying the result inequality by $R$, we obtain
	\begin{align}
	&R\int_{\R^3}(|\D v(\tau)|^2+|\D^2 u(\tau)|^2)\phi^6dx+R\int_{\delta t}^{\tau}\int_{\R^3}(|\D^2v|^2+a|\D^3u|^2)\phi^6dx\\
	&\leq C\varepsilon_1^2\int_{\delta t}^{\tau}\int_{\mathbb R^3}(|\nabla^3u|^2+|\nabla^2v|^2)\phi^6dxds+\frac{C}{R}\int_{\delta t}^{\tau}\int_{\R^3}(p-c(s))^2\phi^4dxds\nonumber\\
	&\quad+R\int_{\R^3}(|\D v(\delta t)|^2+|\D^2 u(\delta)|^2)\phi^6dx+\frac{C\varepsilon_1^2}{R}\int_{\delta t}^{\tau}\int_{\R^3}(|\nabla^2u|+|\nabla v|^2)\phi^4dxds.\nonumber
	\end{align}
	Choosing $\varepsilon_1$ small enough and using \eqref{2.44} lead to \eqref{p-2}. On the other hand, by using \eqref{2.41}, \eqref{2.42} for $r=R$, \eqref{pres2} and \eqref{p-2}, we conclude that for any $x_0\in\R^3$ and $0<\delta t<\tau<t=\sigma R^2$,
	\begin{align*}
	R\int_{B_{\frac{R}{4}}(x_0)}(|\D v(\tau)|^2+|\D^2 u(\tau)|^2)dx\leq\sup_{y\in\R^3}\frac{C}{R}\int_{B_R(y)}|v_0|^2+|\D u_0|^2dx.\nonumber
	\end{align*}
	\end{proof}	

\section{\bf Higher regularity estimates}
In this section, we derive higher regularity estimates as follows.
\begin{lemma}
	Let $(u, v)$ be a solution to \eqref{O-F-1}-\eqref{O-F-3} on $\mathbb R^3\times(0, T)$. There are constants $\varepsilon_1>0$ and $R_0>0$ such that
	\begin{equation*}
	\esssup_{0\leq s\leq T,x\in\mathbb R^3}\int_{B_{R}(x)}|\nabla u(x,s)|^3+|v(x,s)|^3dx< \varepsilon_1^3,\quad\forall R\in[0,R_0].
	\end{equation*}
	Then, for all $x_0\in \mathbb R^3$ and $t\in[\tau,T]$ with $\tau\in (0,T)$, for any $k\geq 0$, it holds that
	\begin{align}\label{3.18}
	&\int_{B_{R}(x_0)}(|\D^{k+1} u(t)|^2+|\D^k v(t)|^2)dx+\int_{\tau}^{t}\int_{B_{R}(x_0)}|\D^{k+2}u|^2+|\D^{k+1}v|^2dxds\\
	&\leq C(\varepsilon_1,\tau,k,T,R).\nonumber
	\end{align}
\end{lemma}
\begin{proof}
	We prove this lemma by induction. It follows from Lemmas 2.1,-2.2  that \eqref{3.18} holds for $k=0,1$. Assume that \eqref{3.18} holds for $l\leq k-1$ with $k\geq 2$;  i.e. for any $t\in(\tau,T)$ and any $l\in[0,k-1]$
	\begin{align}\label{3.19}
	\int_{B_{R}(x_0)}(|\D^{l+1} u|^2+|\D^{l} v|^2)dx+\int_{\tau}^{T}\int_{B_{R}(x_0)}|\D^{l+2}u|^2+|\D^{l+1}v|^2dxds\leq C.
	\end{align}
	By using the Sobolev embedding theorem in  (\ref{3.19}),  we obtain that  for $l\geq 2$,
		\begin{equation}\label{3.22}
		|\D^{l-2}\D u|+|\D^{l-2}v|\leq C, \quad\text{in}\quad B_R(x_0)\times(\tau,T).
		\end{equation}
	
	Now, we prove \eqref{3.19} also holds for $l=k$. Let $\phi\in C_{0}^\infty(B_{2R}(x_0))$ be a cut-off function with $\phi\equiv 1$ on $B_R(x_0)$ and $|\nabla \phi|\leq \frac{C}{R}$ and $|\nabla^2\phi|\leq \frac{C}{R^2}$ for all $R\leq R_0$.
	
	Applying $\D^k$ to \eqref{O-F-1}, multiplying the resulting equation by $\D^kv\phi^2$ and integrating over $\R^3$ give
	\begin{align}\label{3.20}
	&\frac{1}{2}\frac{d}{dt}\int_{\R^3}|\D^k v|^2\phi^2dx+\int_{\R^3}|\D^{k+1}v|^2\phi^2dx\nonumber\\
	&=-\int_{\R^3}\D^k\D_jv^i\D^kv^i\D_j(\phi^2)dx+\int_{\R^3}\D^kp\D^k v^i\D_i(\phi^2)dx\nonumber\\
	&\quad+\int_{\R^3}\D^k(v^j v^i)\D_j(\D^k v\phi^2)dx+\int_{\R^3}\D^k(\D_iu^kW_{p_j^k})\D_j(\D^k v^i\phi^2)dx\\
	&=:I_1+I_2+I_3+I_4.\nonumber
	\end{align}
    It follows from Young's inequality that
    \begin{align}\label{3.23}
    I_1\leq\frac{1}{8}\int_{\R^3}|\D^{k+1}v|^2\phi^2dx+C\int_{\text{spt}\phi}|\D^kv|^2dx.
    \end{align}
	Integration by parts yields
	\begin{align}\label{3.24}
	I_2&=-\int_{\R^3}\D^{k-1}p\D(\D^{k}v^i\D_i(\phi^2))dx\\
	&\leq\frac{1}{8}\int_{\R^3}|\D^{k+1}v|^2\phi^2dx+C\int_{\text{spt}\phi}|\D^{k-1}p|^2+|\D^{k}v|^2dx.\nonumber
	\end{align}
	For $I_3$,  it follows from \eqref{3.19} and \eqref{3.22} that
	\begin{align}\label{3.25}
	I_3&\leq\int_{\R^3}(|v||\D^{k}v|+\sum_{j=1}^{k-1}|\D^jv||\D^{k-j}v|)(|\D^{k+1}v|\phi^2+|\D^{k}v||\D\phi|\phi)\nonumber\\
	&\leq\frac{1}{8}\int_{\R^3}|\D^{k+1}v|^2\phi^2dx+C\int_{\R^3}\left(|v|^2|\D^{k}v|^2+|\D^{[\frac{k}{2}]}v|^2|\D^{\lceil\frac{k}{2}\rceil}v|^2\right)\phi^2dx\\
	&\quad+C\int_{\text{spt}\phi}|\D^{k}v|^2dx+C,\nonumber
	\end{align}
	which is obvious for $k=2,3$ and also holds for $k\geq 4$ due to the fact that  $|\D^{[\frac{k}{2}-1]} v|\leq C$  when $[\frac{k}{2}]-1\leq k-3$.
 Here and in the sequel, $[c]$ denotes the integer part of $c$ and $\lceil c \rceil$ means the smallest integer greater or equal to $c$.
	Similarly,
	\begin{align}\label{3.26}
	I_4&\leq\frac{1}{8}\int_{\R^3}|\D^{k+1}v|^2\phi^2dx+C\int_{\R^3}|\D u|^2|\D^{k+1}u|^2\phi^2dx\\
	&\quad+C\int_{\R^3}|\D^{[\frac{k+2}{2}]}u|^2|\D^{\lceil\frac{k+2}{2}\rceil}u|^2\phi^2dx+C\int_{\text{spt}\phi}|\D^{k}v|^2dx+C.\nonumber
	\end{align}
	Substituting \eqref{3.23}-\eqref{3.26} into \eqref{3.20} yields
	\begin{align}\label{3.27}
	&\frac{1}{2}\frac{d}{dt}\int_{\R^3}|\D^k v|^2\phi^2dx+\frac{1}{2}\int_{\R^3}|\D^{k+1}v|^2\phi^2dx\nonumber\\
	&\leq C\int_{\R^3}(|v|^2|\D^kv|^2+|\D u|^2|\D^{k+1}u|^2)\phi^2dx\nonumber\\
	&\quad+C\int_{\R^3}\left(|\D^{[\frac{k}{2}]}v|^2|\D^{\lceil\frac{k}{2}\rceil}v|^2+|\D^{[\frac{k+2}{2}]}u|^2|\D^{\lceil\frac{k+2}{2}\rceil}u|^2\right)\phi^2dx\\
	&\quad+C\int_{\text{spt}\phi}|\D^{k-1}p|^2+|\D^{k}v|^2dx+C.\nonumber
	\end{align}
	Applying $\D^{k+1}$ to \eqref{O-F-3}, multiplying the resulting equation by $\D^{k+1}u\phi^2$ and integrating over $\R^3$ give
	\begin{align}\label{3.29}
	&\frac{1}{2}\frac{d}{dt}\int_{\R^3}|\D^{k+1}u|^2\phi^2dx-\int_{\R^3}\D^{k+1}\D_\alpha W_{p_\alpha^i}\D^{k+1}u^i\phi^2dx\nonumber\\
	&=-\int_{\R^3}\D^{k+1}\left(\D_\alpha\left(u^kW_{p_\alpha^k}\right)u^i\right)\D^{k+1}u^i\phi^2dx\nonumber\\
	&\quad+\int_{\R^3}\D^{k+1}\left(W_{p_\alpha^k}\D_\alpha u^ku^i\right)\D^{k+1}u^i\phi^2dx\\
	&\quad-\int_{\R^3}\D^{k+1}\left(W_{u^i}-W_{u^i}u^iu^k\right)\D^{k+1}u^i\phi^2dx\nonumber\\
	&\quad-\int_{\R^3}\D^{k+1}(v\cdot\D u^i)\D^{k+1}u^i\phi^2dx\nonumber\\
	&=:J_1+J_2+J_3+J_4.\nonumber
	\end{align}
	Since $W(u,p)$ is quadratic in $u,p$ and
	$$\D_\beta W_{p_\alpha^i}=W_{p_\alpha^ip_\g^j}\D_{\beta\g}^2 u^j+W_{p_\alpha^iu^j}\D_\beta u^j,$$
	it follows from integration by parts and \eqref{uni-ell} that
	\begin{align}\label{3.30}
	&-\int_{\R^3}\D^{k+1}\D_\alpha W_{p_\alpha^i}\D^{k+1}u^i\phi^2dx\nonumber\\
	&=\int_{\R^3}\D^{k+1} W_{p_\alpha^i}\D^{k+1}\D_\alpha u^i\phi^2dx+\int_{\R^3}\D^{k+1}W_{p_\alpha^i}\D^{k+1} u^i\D_\alpha(\phi^2)dx\nonumber\\
	&=\int_{\R^3}W_{p_\alpha^ip_\g^j}\D^k\D_{\beta\g}^2 u^j\D^k\D_{\beta\alpha}^2u^i\phi^2dx+\int_{\R^3}\D^{k+1}W_{p_\alpha^i}\D^{k+1} u^i\D_\alpha(\phi^2)dx\nonumber\\
	&\quad+\int_{\R^3}\left(\D^k\left(W_{p_\alpha^ip_\g^j}\D_{\beta\g}^2 u^j\right)-W_{p_\alpha^ip_\g^j}\D^k\D_{\beta\g}^2 u^j\right)\D^k\D_{\beta\alpha}^2u^i\phi^2dx\\
	&\quad+\int_{\R^3}\D^k\left(W_{p_\alpha^iu^j}\D_\beta u^j\right)\D^k\D_{\beta\alpha}^2u^i\phi^2dx\nonumber\\
	&\geq \frac{7a}{8}\int_{\R^3}|\D^{k+2}u|^2\phi^2dx-C\int_{\R^3}|\D u|^2|\D^{[\frac{k+1}{2}]}u|^2|\D^{\lceil\frac{k+1}{2}\rceil}u|^2\phi^2dx\nonumber\\
	&\quad-C\int_{\R^3}\left(|\D u|^2|\D^{k+1}u|^2+|\D^{[\frac{k+2}{2}]}u|^2|\D^{\lceil\frac{k+2}{2}\rceil}u|^2\right)\phi^2dx\nonumber\\
	&\quad-C\int_{\text{spt}\phi}|\D^{k+1}u|^2\phi^2dx-C,\nonumber
	\end{align}
	which is obvious for the case $k=2$ and also holds true for $k\geq 3$ since, in this case, $k-2\geq[\frac{k}{2}]$ implies $|\D^{[\frac{k}{2}]} u|\leq C$.
	
	The right hand side of \eqref{3.29} can be estimated as follows.
	
	Since $|u^i\D^{k}\Delta u^i|\leq\sum_{l=1}^{k+1}|\D^lu||\D^{k+2-l}|$ due to $|u|^2=1$ and
	$$\D^l(u^kW_{p_\alpha^k})=\D^{\lambda_1}u\#\D^{\lambda_2}u+\D^{\mu_1}u\#\D^{\mu_2}u\#\D^{\mu_3}u\#\D^{\mu_4}u$$
	where $\lambda_1+\lambda_2=\mu_1+\mu_2+\mu_3+\mu_4=l$ and $\#$ denotes the multi-linear map with constant coefficients, it follows from integration by parts and Young's inequality that
	\begin{align}\label{3.31}
	J_1&=\int_{\R^3}\D^{k}\left(\D_\alpha(u^kW_{p_k^\alpha})\right)u^i\D^{k}\Delta u^i\phi^2dx\nonumber\\
	&\quad+\int_{\R^3}\D^{k}\left(\D_\alpha(u^kW_{p_k^\alpha})u^i\right)\D^{k}\D_\beta u^i\D_\beta(\phi^2)dx\nonumber\\
	&\quad+\int_{\R^3}\left[\D^{k}\left(\D_\alpha(u^kW_{p_k^\alpha})u^i\right)-\D^{k}\left(\D_\alpha(u^kW_{p_k^\alpha})\right)u^i\right]\D^{k}\Delta u^i\phi^2dx\nonumber\\
	&\leq\frac{a}{8}\int_{\R^3}|\D^{k+2}u|^2\phi^2dx+C\int_{\R^3}|\D u|^2|\D^{[\frac{k+1}{2}]}u|^2|\D^{\lceil\frac{k+1}{2}\rceil}u|^2\phi^2dx\\
	&\quad+C\int_{\R^3}\left(|\D u|^2|\D^{k+1}u|^2+|\D^{[\frac{k+2}{2}]}u|^2|\D^{\lceil\frac{k+2}{2}\rceil}u|^2\right)\phi^2dx\nonumber\\
	&\quad+\int_{\text{spt}\phi}|\D^{k+1}u|^2dx+C,\nonumber
	\end{align}
	which is obvious for the case $k=2$ and also holds true for $k\geq 3$ due to the fact that  $|\D^{[\frac{k}{2}]} u|\leq C$  when $[\frac{k}{2}]\leq k-2$. Similarly, noting that $W_{u^i}=u\#\D u\#\D u$, it is easy to obtain that
	\begin{align}\label{3.32}
	J_2+J_3\leq&\frac{a}{8}\int_{\R^3}|\D^{k+2}u|^2\phi^2dx+C\int_{\R^3}|\D u|^2|\D^{[\frac{k+1}{2}]}u|^2|\D^{\lceil\frac{k+1}{2}\rceil}u|^2\phi^2dx\nonumber\\
	&+C\int_{\R^3}\left(|\D u|^2|\D^{k+1}u|^2+|\D^{[\frac{k+2}{2}]}u|^2|\D^{\lceil\frac{k+2}{2}\rceil}u|^2\right)\phi^2dx\\
	&+\int_{\text{spt}\phi}|\D^{k+1}u|^2dx+C\nonumber
	\end{align}
	and
	\begin{align}\label{3.33}
	J_4\leq&\frac{a}{8}\int_{\R^3}|\D^{k+2}u|^2\phi^2dx+C\int_{\R^3}|v|^2|\D^{k+1}u|^2\phi^2dx\nonumber\\
	&+C\int_{\R^3}\left(|\D^{[\frac{k}{2}]}v|^2|\D^{\lceil\frac{k}{2}\rceil+1}u|^2+|\D^{[\frac{k+2}{2}]}u|^2|\D^{\lceil\frac{k}{2}\rceil}v|^2\right)\phi^2dx+C.
	\end{align}
	Substituting \eqref{3.30}-\eqref{3.33} into \eqref{3.29}, we obtain
	\begin{align}\label{3.34}
	&\frac{1}{2}\frac{d}{dt}\int_{\R^3}|\D^{k+1}u|^2\phi^2dx+\frac{a}{2}\int_{\R^3}|\D^{k+2}u|^2\phi^2dx\nonumber\\
	&\leq C\int_{\R^3}|v|^2|\D^{k+1}u|^2\phi^2dx+C\int_{\R^3}|\D u|^2|\D^{[\frac{k+1}{2}]}u|^2|\D^{\lceil\frac{k+1}{2}\rceil}u|^2\phi^2dx\nonumber\\
	&\quad+C\int_{\R^3}\left(|\D u|^2|\D^{k+1}u|^2+|\D^{[\frac{k+2}{2}]}u|^2|\D^{\lceil\frac{k+2}{2}\rceil}u|^2\right)\phi^2dx\\
	&\quad+C\int_{\R^3}\left(|\D^{[\frac{k}{2}]}v|^2|\D^{\lceil\frac{k}{2}\rceil+1}u|^2+|\D^{[\frac{k+2}{2}]}u|^2|\D^{\lceil\frac{k}{2}\rceil}v|^2\right)\phi^2dx\nonumber\\
	&\quad+\int_{\text{spt}\phi}|\D^{k+1}u|^2dx+C.\nonumber
	\end{align}
	Summing \eqref{3.27} with \eqref{3.34} gives
	\begin{align}\label{3.37}
	&\frac{d}{dt}\int_{\R^3}\left(|\D^kv|^2+|\D^{k+1}u|^2\right)\phi^2dx+\int_{\R^3}(|\D^{k+1}v|^2+a|\D^{k+2}u|^2)\phi^2dx\nonumber\\
	&\leq C\int_{\R^3}|\D u|^2|\D^{[\frac{k+1}{2}]}u|^2|\D^{\lceil\frac{k+1}{2}\rceil}u|^2\phi^2dx\nonumber\\
	&+C\int_{\R^3}\left(|\D^{[\frac{k}{2}]}v|^2|\D^{\lceil\frac{k}{2}\rceil+1}u|^2+|\D^{[\frac{k+2}{2}]}u|^2|\D^{\lceil\frac{k}{2}\rceil}v|^2\right)\phi^2dx\\
	&+C\int_{\R^3}\left(|\D^{[\frac{k}{2}]}v|^2|\D^{\lceil\frac{k}{2}\rceil}v|^2+|\D^{[\frac{k+2}{2}]}u|^2|\D^{\lceil\frac{k+2}{2}\rceil}u|^2\right)\phi^2dx\nonumber\\
	&+C\int_{\R^3}(|v|^2+|\D u|^2)(|\D^kv|^2+|\D^{k+1}u|^2)\phi^2dx\nonumber\\
	&+C\int_{\text{spt}\phi}\left(|\D^{k+1}u|^2+|\D^kv|^2+|\D^{k-1}p|^2\right)dx+C=:\sum_{i=1}^{5}K_i+C.\nonumber
	\end{align}
	It remains to estimate $K_1, \cdots, K_5$. Since $k-3\geq [\frac{k}{2}]$ for any $k\geq 5$, \eqref{3.19} and \eqref{3.20} yield $\sum_{i=1}^{3}K_i\leq C,\quad\forall k\geq 5.$ Hence, it only need to estimate $K_1$,$K_2$ and $K_3$ for $k=2,3,4$. For $k=2$,
	\begin{align}
	\sum_{i=1}^{3}K_i\leq C\int_{\R^3}(|\D v|^4+|\D^2u|^4)\phi^2dx.
	\end{align}
	It follows from integration by parts that
	\begin{align}\label{3.38}
	&C\int_{\R^3}(|\D v|^4+|\D^2u|^4)\phi^2dx\nonumber\\
	&\leq C\int_{\R^3}(|v||\D v|^2|\D^2v|+|\D u||\D^3 u|^2|\D^2u|^2)\phi^2dx\nonumber\\
	&\quad+C\int_{\R^3}(|v||\D v|^3+|\D u||\D^2u|^3)|\D\phi|\phi dx\\
	&\leq\frac{C}{2}\int_{\R^3}(|\D v|^4+|\D^2u|^2)\phi^2dx+\frac{C}{2}\int_{\R^3}(|v|^2|\D^2v|^2+|\D u|^2|\D^3u|^2)\phi^2dx\nonumber\\
	&\quad+\frac{C}{2}\int_{\R^3}(|v|^2|\D v|^2+|\D u|^2|\D^2u|^2)|\D\phi|^2dx.\nonumber
	\end{align}
	Then, by using Sobolev embedding theorem and \eqref{3.19}, one has
	\begin{align}\label{3.39}
	\sum_{i=1}^{3}K_i&\leq C\int_{\R^3}(|v|^2+|\D u|^2)\left((|\D^2 v|^2+|\D^3u|^2)\phi^2+(|\D v|^2+|\D^2u|^2)|\D\phi^2|\right)dx\nonumber\\
	&\leq C\varepsilon_1^2\int_{\R^3}(|\D^3v|^2+|\D^4u|^2)\phi^2dx+C.
	\end{align}
	For $k=3$, \eqref{3.20} implies $|v|+|\D u|\leq C$. It follows from integration by parts and \eqref{3.19} that
	\begin{align*}
	\sum_{i=1}^{3}K_i\leq&C\int_{\R^3}\left(|\D v|^2+|\D^2u|^2\right)\left(|\D^2v|^2+|\D^3 u|^2\right)\phi^2dx+C\nonumber\\
	&\leq C\int_{\R^3}\left(|v||\D^2v|+|\D u||\D^3u|\right)\left(|\D^2v|^2+|\D^3 u|^2\right)\phi^2dx\nonumber\\
	&\quad+C\int_{\R^3}\left(|v||\D v|+|\D u||\D^2u|\right)\left(|\D^2v||\D^3v|+|\D^3 u||\D^4u|\right)\phi^2dx\nonumber\\
	&\quad+C\int_{\R^3}\left(|v||\D v|+|\D u||\D^2u|\right)\left(|\D^2v|^2+|\D^3 u|^2\right)\D\phi\phi dx+C\nonumber\\
	&\leq\frac{C}{2}\int_{\R^3}\left(|\D v|^2+|\D^2u|^2\right)\left(|\D^2v|^2+|\D^3 u|^2\right)\phi^2dx\\
	&\quad+C\int_{\R^3}\left(|\D^3v|^2+|\D^4u|^2\right)\phi^2+\left(|\D^2v|^2+|\D^3u|^2|\D\phi|^2\right)dx+C,\nonumber
	\end{align*}
	which implies that
	\begin{equation}\label{3.41}
	\sum_{i=1}^{3}K_i\leq C,\quad\text{for}\quad k=3.
	\end{equation}
	Similarly, since $|\D v|+|\D^2u|\leq C$  for $k=4$, one has
	\begin{align}\label{3.42}
	\sum_{i=1}^{3}K_i&\leq C\int_{\R^3}|\D^2v|^2|\D^3u|^2\phi^2dx+C\int_{\R^3}(|\D^3u|^4+|\D^2 v|^4)\phi^2dx\nonumber\\
	&\leq \frac{C}{2}\int_{\R^3}|\D^2v|^2|\D^3u|^2\phi^2dx+\frac{C}{2}\int_{\R^3}(|\D^3u|^4+|\D^2 v|^4)\phi^2dx\\
	&\quad+C\int_{\R^3}(|\D^3v|^2+|\D^4u|^2)\phi^2dx+C\int_{\R^3}(|\D^2v|^2+|\D^3u|^2)|\D\phi|^2dx\nonumber
	\end{align}
	which implies that
	\begin{align}
	\sum_{i=1}^{3}K_i\leq C,\quad\text{for}\quad k=4.
	\end{align}
	To estimate $K_4$, it follows from the Sobolev embedding theorem that
	\begin{align}\label{3.43}
	&K_4\leq \left(\int_{B_R(x_0)}(|v|^3+|\D u|^3)\right)^\frac{2}{3}\left(\int_{\R^3}(|\D^kv|^6+|\D^{k+1}u|^6)\phi^6\right)^\frac{1}{3}\\
	&\leq C\varepsilon_1^2\int_{\R^3}(|\D^{k+1}v|^2+|\D^{k+2}u|^2)\phi^2+(|\D^{k}v|^2+|\D^{k+1}u|^2)\D\phi^2dx.\nonumber
	\end{align}
	To estimate $K_5$, the assumption \eqref{2.19} yields
	\begin{align}\label{3.44}
	\int_{\tau}^{T}\int_{\text{spt}\phi}|\D^kv|^2+|\D^{k+1}u|^2dx\leq C.
	\end{align}
	For the pressure term in $K_5$, since $\D p$ solves
	\begin{equation*}
	-\Delta\D^{k-1} p=\nabla_{ij}\D^{k-1}[\nabla_iu^kW_{p_j^k}(u,\nabla u)+v^jv^i+\sigma^L_{ij}],\quad\text{on}~~\mathbb{R}^3\times[0,T],
	\end{equation*}
	it follows from standard elliptic estimates that
	\begin{align}\label{3.45}
	\int_{\tau}^{T}\int_{\text{spt}\phi}|\D^{k-1} p|^2dxdt&\leq \sup_{\tau\leq t\leq T}\int_{2\text{spt}\phi}(|\D^{k-1}v|^2+|\D^k u|^2)\nonumber\\
	&+\int_{\tau}^{T}\int_{2\text{spt}\phi}(|p-c|^2+|\D^kv|^2+|\D^{k+1}u|^2)\leq C.
	\end{align}
	Integrating over $[\tau,t]$ and substituting \eqref{3.38}-\eqref{3.45} into \eqref{3.37} yield that \eqref{3.19} holds for $l=k$ by choosing $\varepsilon_1$ small enough. This complete a proof of this lemma.
\end{proof}

\section{\bf Existence of $L_{uloc}^3$-solutions}

In this section, we prove existence of $L_{uloc}^3$-solutions to \eqref{O-F-1}-\eqref{O-F-3} in $\R^3$.  We first approximate $(v_0,u_0)$ satisfying \eqref{Ini} by smooth data in the following lemma, which is stated and proved in Lemma 5.1 of \cite{HW}.
\begin{lemma}\label{appro.-lem}
	For a sufficiently small $\varepsilon_0>0$, let $(v_0,u_0):\R^3\rightarrow\R^3\times S^2$ satisfy \eqref{Ini}. Then,  there exists a sequence of functions $$(v_0^k, u_0^k)\in C^\infty(\R^3;\R^3\times S^2)\cap\cap_{q=2}^3 (L^q(\R^3,\R^3)\times \dot{W}^{1,q}_b(\R^3,S^2))$$ such that the following properties hold:
	\begin{itemize}
		\item[(i)] $\D\cdot v_0^k=0$ in $\R^3$ for all $k\geq 1$.
		\item[(ii)] There exist $C_0>0$ and $k_0>1$ such that for any $k\geq k_0$,
		$$\|(v_0^k,u_0^k)\|_{L^3_{R_0}(\R^3)}\leq C_0\varepsilon_0.$$
		\item[(iii)] $(v_0^k,u_0^k,\D u_0^k)\rightarrow (v_0,u_0,\D u_0)$ in $L^q_{loc}(\R^3)$ for $q=2,3$, as $k\rightarrow +\infty$.
	\end{itemize}
\end{lemma}
\begin{proof}[{\bf Proof of Theorem \ref{thm1}:}]
For initial data $(v_0,u_0)$ satisfying \eqref{Ini}, it follows from Lemma \ref{appro.-lem} that there are approximation smooth data $(v_0^k,u_0^k)$ for $(v_0,u_0)$ such that
	\begin{equation}\label{ini-appr}
	\sup_{x_0\in\R^3}\int_{B_{R_0}(x_0)}|\D u_0^k|^3+|v_0^k|^3dx\leq C_0^3\varepsilon_0^3\leq\frac{\varepsilon_1^3}{2}.
	\end{equation}
	where $\varepsilon_1=C\varepsilon_0$ to be determined later.
	Then, by using the result in Hong-Li-Xin \cite{HLX} or Wang-Wang \cite{WW}, there exists a unique local smooth solution $(v^k, u^k):\R^3\times[0,T^k]\rightarrow \R^3\times S^2$ to \eqref{O-F-1}-\eqref{O-F-3} with smooth initial data $(v_0^k, u_0^k)$ such that
	$$(v^k,u^k)\in C^\infty(\R^3\times(0,T^k))\cap C([0,T^k];H^1(\R^3,\R^3)\times H_b^2(\R^3,S^2)).$$
	Therefore, there exists a maximal time $t_k\in(0,T^k]$ such that
	\begin{equation}\label{assumption}
	\sup_{0\leq s\leq t_k, x_0\in\mathbb R^3}\int_{B_{R_0}(x_0)}|\nabla u^k|^3+|v^k|^3dx\leq \varepsilon_1^3.
	\end{equation}
	
	We claim that there exists a  uniform lower bound  of $t_k$ such that $t_k\geq \sigma R_0^2$ for some small constant $\sigma>0$.
	
	The proof of claim relies on delicate covering arguments, the estimates in Proposition \ref{prop1} and the following interpolation inequality (c.f. \cite{CKN})
\begin{align}\label{inter}
\int_{B_r}|f|^3dx\leq C(\int_{B_r}|f|^2dx)^{\frac{3}{4}}(\int_{B_r}|\nabla f|^2dx)^\frac{3}{4}+C(\frac{1}{r}\int_{B_r}|f|^2dx)^{\frac{3}{2}}.
\end{align}
We prove the claim by contradiction. Assume that $t_k=\sigma R^2<\sigma R_0^2$. In order to apply covering arguments, let $\R^3$ be covered by infinitely many balls of radius $\frac{R}{4}$, that is, $\R^3=\bigcup\limits_{x_i\in\mathcal{I}}B_{\frac{R}{4}}(x_i)$ with an index set $\mathcal{I}$ being all the centers of covering balls. Then, $$B_{R_0}(x_0)\subset\bigcup\limits_{x_i\in \mathcal{I}_{x_0}}B_{\frac{R}{4}}(x_i),$$ where $\mathcal{I}_{x_0}=\{x_i\in\mathcal{I}\big| B_{R_0}(x_0)\cap B_{\frac{R}{4}}(x_i)\neq\emptyset\}$. It is clear that the number of elements in $\mathcal{I}_{x_0}$ is bounded by $|\mathcal{I}_{x_0}|\leq C\left(\frac{R_0}{R}\right)^3$. In particular, $B_R(x_0)$ can be covered by a fixed number (independent of $R$ and $x_0$) of balls $B_{\frac{R}{2}}(x_i)$. Then using \eqref{inter}  and H\"{o}lder's inequality,  we have, for any $x_0\in\R^3$,
	\begin{align}\label{4.4}
	&\int_{B_{R_0}(x_0)}|v^k(t_k)|^3+|\D u^k(t_k)|^3dx\nonumber\\
	&\leq C\sum_{x_i\in \mathcal{I}_{x_0}}\int_{B_{\frac{R}{4}}(x_i)}|v^k(t_k)|^3+|\D u^k(t_k)|^3dx\nonumber\\
	&\leq C\sum_{x_i\in \mathcal{I}_{x_0}}\left(\frac{1}{R}\int_{B_{\frac{R}{4}}(x_i)}|v^k(t_k)|^2+|\D u^k(t_k)|^2dx\right)^\frac{3}{2}\\
	&\quad+C\sum_{x_i\in \mathcal{I}_{x_0}}\left(R\int_{B_{\frac{R}{4}}(x_i)}|\D v^k(t_k)|^2+|\D^2 u^k(t_k)|^2dx\right)^\frac{3}{2}\nonumber\\
	&=:I+J.\nonumber
	\end{align}
	For $I$, it follows from \eqref{p-1} that
	 \begin{align}\label{4.5}
	 I&\leq C\sum_{x_i\in\mathcal{I}_{x_0}}\left(\frac{1}{R}\int_{B_{\frac{R}{2}}(x_i)}|v_0^k|^2+|\D u_0^k|^2dx\right)^\frac{3}{2}\nonumber\\
	 &\quad+C\sum_{x_i\in\mathcal{I}_{x_0}}\left(\frac{1}{R^3}\int_{0}^{t_k}\int_{B_{\frac{R}{2}}(x_i)}|v^k|^2+|\D u^k|^2dxds\right)^\frac{3}{2}\\
	 &\quad+C\sum_{x_i\in\mathcal{I}_{x_0}}\left(\frac{1}{R}\int_{0}^{t_k}\int_{\R^3}(p-c(s))v\cdot\D\phi_i\phi_idxds\right)^\frac{3}{2}\nonumber\\
	 &=:I_1+I_2+I_3,\nonumber
	 \end{align}
	 where $\phi_i$ is a cut-off function compactly supported in $B_{\frac{R}{2}}(x_i)$ and $\phi_i\equiv 1$ on $B_{\frac{R}{4}}(x_i)$.
	 Then, by using H\"{o}lder's inequality, the standard covering argument and \eqref{ini-appr}, we obtain
	 \begin{align}\label{4.6}
	 I_1&\leq C|B_1|^\frac{1}{2}\sum_{x_i\in \mathcal{I}_{x_0}}\int_{B_{\frac{R}{2}}(x_i)}|v_0^k|^3+|\D u_0^k|^3dx\\
	 &\leq C\int_{B_{2R_0}(x_0)}|v_0^k|^3+|\D u_0^k|^3dx\nonumber\\
	 &\leq C\sup_{y\in\R^3}\int_{B_{R_0}(y)}|v_0^k|^3+|\D u_0^k|^3dx\leq C_1\varepsilon_0^3,\nonumber
	 \end{align}
	 where we have used $\bigcup\limits_{x_i\in\mathcal{I}_{x_0}}B_{\frac{R}{2}}(x_i)\subset B_{2R_0}(x_0)$.
	 Similarly, using \eqref{assumption} and $t_k=\sigma R^2$, we have
	 \begin{align}\label{4.7}
	 I_2&\leq CR^{-\frac{9}{2}}\sum_{x_i\in\mathcal{I}_{x_0}}\left(\int_{0}^{t_k}\int_{B_{\frac{R}{2}}(x_i)}|v^k|^3+|\D u^k|^3dxds\right)\cdot\left(\int_{0}^{t_k}\int_{B_{\frac{R}{2}}(x_i)}1dxds\right)^\frac{1}{2}\\
	 &\leq C|B_1|^\frac{1}{2}t_k^\frac{1}{2}R^{-3}\int_{0}^{t_k}\left(\sum_{x_i\in\mathcal{I}_{x_0}}\int_{B_{\frac{R}{2}}(x_i)}|v^k|^3+|\D u^k|^3dx\right)ds\nonumber\\
	 &\leq Ct_k^\frac{1}{2}R^{-3}\int_{0}^{t_k}\int_{B_{2R_0}(x_0)}|v^k|^3+|\D u^k|^3dxds\nonumber\\
	 &\leq Ct_k^\frac{3}{2}R^{-3}\sup_{0\leq s\leq t_k,y\in\R^3}\int_{B_{R_0}(y)}|v^k|^3+|\D u^k|^3dx\leq C_1\sigma^\frac{3}{2}\varepsilon_1^3\nonumber
	 \end{align}
	 and it follows from \eqref{assumption} and \eqref{pres1} that
	 \begin{align}\label{4.8}
	 I_3&\leq CR^{-3}\sum_{x_i\in \mathcal{I}_{x_0}}\left(\int_{0}^{t_k}\int_{B_{\frac{R}{2}(x_i)}}(p-c(s))^{\frac{3}{2}}\phi_i^\frac{3}{2}dxds\right)\left(\int_{0}^{t_k}\int_{B_{\frac{R}{2}(x_i)}}|v|^3dx\right)^\frac{1}{2}\nonumber\\
	 &\leq C\varepsilon_1^\frac{3}{2}t_k^\frac{1}{2}R^{-3}\sum_{x_i\in\mathcal{I}_{x_0}}\left(\int_{0}^{t_k}\int_{B_{\frac{R}{2}}(x_i)}(p-c(s))^\frac{3}{2}\phi_i^\frac{3}{2}dxds\right)\nonumber\\
	 &\leq C\varepsilon_1^\frac{3}{2}t_k^\frac{1}{2}R^{-3}\int_{0}^{t_k}\int_{B_{2R_0(x_0)}}(p-c(s))^\frac{3}{2}\phi^\frac{3}{2}dxds\\
	 &\leq C\varepsilon_1^\frac{9}{2}t^\frac{3}{2}_kR^{-3}\leq C_1\sigma^\frac{3}{2}\varepsilon_1^\frac{9}{2},\nonumber
	 \end{align}
	 where $\phi$ is the cut-off function compactly supported in $B_{2R_0}(x_0)$. Substituting \eqref{4.6}-\eqref{4.8} into \eqref{4.5} and choosing $\varepsilon_0$, $\sigma$ small enough such that $2C_1^\frac{1}{3}\varepsilon_0=\varepsilon_1<1$ and $C_1\sigma^\frac{3}{2}<\frac{1}{16}$, we have
	 \begin{equation}\label{4.9}
	 I\leq C_1\varepsilon_0^3+C_1\sigma^\frac{3}{2}(\varepsilon_1^3+\varepsilon_1^\frac{9}{2})\leq \frac{\varepsilon_1^3}{4}.
	 \end{equation}
	 To estimate $J$,  \eqref{p-2} yields that for any $\delta t_k\leq t\leq t_k$,
	 \begin{align}\label{4.10}
	 &C\sum_{x_i\in \mathcal{I}_{x_0}}\left(R\int_{B_{\frac{R}{4}}(x_i)}|\D v^k(t)|^2+|\D^2 u^k(t)|^2dx\right)^\frac{3}{2}\\
	 &\leq C\sum_{x_i\in\mathcal{I}_{x_0}}\left(\frac{1}{R}\int_{B_R(x_i)}|v_0|^2+|\D u_0|^2dx\right)^\frac{3}{2}\nonumber\\
	 &+C\sum_{x_i\in\mathcal{I}_{x_0}}\left(\frac{1}{R^3}\int_{0}^{t}\int_{B_{\frac{R}{2}}(x_i)}|v^k|^2+|\D u^k|^2dxds\right)^\frac{3}{2}\nonumber\\
	 &+C\sum_{x_i\in\mathcal{I}_{x_0}}\left(\frac{1}{R}\int_{0}^{t}\int_{B_{\frac{R}{2}}(x_i)}(p-c(s))v\cdot\D\phi_i\phi_idxds\right)^\frac{3}{2}\nonumber\\
	 &+C\sum_{x_i\in\mathcal{I}_{x_0}}\left(\frac{1}{R}\int_{\delta t_k}^{t}\int_{B_{\frac{R}{2}}(x_i)}(p-c(s))^2\phi_i^4dxds\right)^\frac{3}{2}\nonumber\\
	 &=:J_1+J_2+J_3+J_4,\nonumber
	 \end{align}
	  where $\phi_i$ is a cut-off function compactly supported in $B_{\frac{R}{2}}(x_i)$ and $\phi_i\equiv 1$ on $B_{\frac{R}{4}}(x_i)$.
	 By the same arguments as in \eqref{4.6}-\eqref{4.8}, one has
	 \begin{align}\label{4.11}
	 J_1+J_2+J_3\leq \frac{\varepsilon_1^3}{8}.
	 \end{align}
	 It follows from H\"{o}lder's inequality, \eqref{assumption} and \eqref{pres3} that
	 \begin{align}\label{4.12}
	 J_4&\leq CR^{-\frac{3}{2}}\sum_{x_i\in\mathcal{I}_{x_0}}\left(\int_{\delta t_k}^{t}\int_{B_{\frac{R}{2}}(x_i)}(p-c)^3\phi_i^6dxds\right)\cdot\left(\int_{\delta t_k}^{t}\int_{B_{\frac{R}{2}}(x_i)}1dxds\right)^\frac{1}{2}\\
	 &\leq C(1-\delta)t_k^\frac{1}{2}\sum_{x_i\in\mathcal{I}_{x_0}}\int_{\delta t_k}^{t}\int_{B_{\frac{R}{2}}(x_i)}(p-c)^3\phi_i^6dxds\nonumber\\
	 &\leq C(1-\delta)^2\frac{t_k^\frac{3}{2}}{R_0^3}\varepsilon_1^6+C(1-\delta)t_k^\frac{1}{2}\sum_{x_i\in\mathcal{I}_{x_0}}\int_{\delta t_k}^{t}\int_{B_{\frac{R}{2}}(x_i)}(|v^k|^6+|\D u^k|^6)\phi_i^6dxds.\nonumber
	 \end{align}
	 For the second term in the last inequality of \eqref{4.12}, using the Sobolev embedding theorem, one has
	 \begin{align}\label{3.32-2}
	 &Ct_k^\frac{1}{2}\sum_{x_i\in\mathcal{I}_{x_0}}\int_{\delta t_k}^{t}\int_{B_{\frac{R}{2}}(x_i)}(|v^k|^6+|\D u^k|^6)\phi_i^6dxds\nonumber\\
	 &\leq Ct_k^\frac{1}{2}\sum_{x_i\in \mathcal{I}_{x_0}}\int_{\delta t_k}^{t}\left(\int_{\R^3}(|v^k|^2+|\D u^k|^2)|\D\phi_i|^2dx\right)^3ds\\
	 &\quad+Ct_k^\frac{1}{2}\sum_{x_i\in \mathcal{I}_{x_0}}\int_{\delta t_k}^{t}\left(\int_{\R^3}(|\D v^k|^2+|\D^2u^k|^2)\phi_i^2dx\right)^3ds=:K_1+K_2.\nonumber
	 \end{align}
	 To estimate $K_1$, it follows from H\"{o}lder's inequality, \eqref{assumption} and a standard covering argument that
	 \begin{align}\label{3.33-3}
	 K_1&\leq \frac{Ct_k^\frac{1}{2}}{R^6}\sum_{x_i\in \mathcal{I}_{x_0}}\int_{\delta t_k}^{t}\left(\int_{B_{\frac{R}{2}(x_i)}}|v^k|^3+|\D u^k|^3dx\right)^2\cdot\left(\int_{B_{\frac{R}{2}(x_i)}}1dx\right)ds\nonumber\\
	 &\leq\frac{C\varepsilon_1^3t_k^\frac{1}{2}}{R^3}\int_{\delta t_k}^{t}\left(\sum_{x_i\in \mathcal{I}_{x_0}}\int_{B_{\frac{R}{2}(x_i)}}|v^k|^3+|\D u^k|^3dx\right)ds\\
	 &\leq \frac{C\varepsilon_1^3t_k^\frac{1}{2}}{R^3}\int_{\delta t_k}^{t}\int_{B_{2R_0}(x_0)}|v^k|^3+|\D u^k|^3dxds\nonumber\\
	 &\leq \frac{C\varepsilon_1^3t_k^\frac{3}{2}}{R^3}\sup_{\delta t_k\leq s\leq t,y\in\R^3}\int_{B_{R_0}(y)}|v^k|^3+|\D u^k|^3dx\leq C_2\sigma^\frac{3}{2}\varepsilon_1^6.\nonumber
	 \end{align}
To estimate $K_2$, it follows from \eqref{p-2-2} that
\begin{align*}
&\sup_{\delta t_k\leq s\leq t,x_i\in\R^3}R\int_{B_{\frac{R}{4}}(x_i)}(|\D v^k|^2+|\D^2 u^k|)dx\\
&\leq \sup_{y\in\R^3}\frac{C}{R}\int_{B_{R}(y)}|v_0|^2+|\D u_0|^2dx\leq C|B_1|^\frac{1}{3}\varepsilon_0^2,
\end{align*}
 which implies
  \begin{align}\label{3.34-4}
 K_2&\leq Ct_k^\frac{1}{2}R^{-3}\int_{\delta t_k}^{t}\sum_{x_i\in \mathcal{I}_{x_0}}\left(R\int_{B_{\frac{R}{2}}(x_i)}(|\D v^k|^2+|\D^2 u^k|^2)dx\right)^3dx\\
 &\leq C\varepsilon_0^2t_k^\frac{1}{2}R^{-3}\int_{\delta t_k}^{t}\sum_{x_i\in \mathcal{I}_{x_0}}\left(R\int_{B_{\frac{R}{2}}(x_i)}(|\D v^k|^2+|\D^2 u^k|^2)dx\right)^\frac{3}{2}ds.\nonumber
 \end{align}
Substituting \eqref{4.11}-\eqref{3.34-4} into \eqref{4.10} and choosing $\sigma$ small such that $C_2\sigma^\frac{3}{2}\leq\frac{1}{8}$, we obtain, for any $x_0\in\R^3$ and $\delta t_k\leq t\leq t_k$,
	 \begin{align}\label{4.16}
	 &\sum_{x_i\in \mathcal{I}_{x_0}}C\left(R\int_{B_{\frac{R}{4}}(x_i)}|\D v^k(t)|^2+|\D^2 u^k(t)|^2dx\right)^\frac{3}{2}\\
	 &\leq C\varepsilon_0^2t_k^\frac{1}{2}R^{-3}\int_{\delta t_k}^{t}\sum_{x_i\in \mathcal{I}_{x_0}}\left(R\int_{B_{\frac{R}{2}}(x_i)}(|\D v^k|^2+|\D^2 u^k|^2)dx\right)^\frac{3}{2}ds+\frac{\varepsilon_1^3}{4}.\nonumber
	 \end{align}
	 It remains to estimate the first term on right hand side of \eqref{4.16}. Indeed, for any $x_i\in\mathcal{I}_{x_0}$, one has the covering $$B_{\frac{R}{2}(x_i)}\subset\bigcup\limits_{x_{ij}\in \mathcal{J}_{x_i}}B_{\frac{R}{4}}(x_{ij}),$$ where $\mathcal{J}_{x_i}=\{x_{ij}\in \mathcal{I}\big|B_{\frac{R}{2}}(x_i)\cap B_{\frac{R}{4}}(x_{ij})\neq\emptyset\}$ and $|J_{x_i}|\leq C$ is independent of $R$ and $x_i$. Denote $$\bigcup\limits_{x_i\in\mathcal{I}_{x_0}}B_{\frac{R}{2}}(x_i)\subset\bigcup\limits_{y_m\in \mathcal{M}_{x_0}} B_{\frac{R}{4}}(y_m),$$ where $\mathcal{M}_{x_0}=\bigcup\limits_{x_i\in \mathcal{I}_{x_0}}J_{x_i}\subset\mathcal{I}.$ Moreover, we have the following chain of coverings
	 $$B_{R_0}(x_0)\subset\bigcup\limits_{x_i\in\mathcal{I}_{x_0}}B_{\frac{R}{4}}(x_i)\subset\bigcup\limits_{x_i\in\mathcal{I}_{x_0}}B_{\frac{R}{2}}(x_i)\subset\bigcup\limits_{y_m\in \mathcal{M}_{x_0}} B_{\frac{R}{4}}(y_m)\subset B_{2R_0}(x_0).$$
	 Next, we divide all the balls of radius $\frac{R}{4}$ with centers $y_m\in\mathcal{M}_{x_0}$ into classes as follows. Note that the ball $B_{2R_0}(x_0)$ has the covering $B_{2R_0}(x_0)\subset\bigcup\limits_{l=1}^{L}B_{R_0}(z_l)$ for a fixed $L$, which is  independent of $R_0$ and $x_0$, and  define
	 $$\mathcal{A}_{x_0}^l=\{y_m\in \mathcal{M}_{x_0}\subset \mathcal{I}\big| B_{\frac{R}{4}}(y_m)\cap B_{R_0}(z_l)\neq\emptyset\}.$$
	 Then we have
	$$B_{R_0}(z_l)\subset\bigcup\limits_{y_m\in\mathcal{A}_{x_0}^l}B_{\frac{R}{4}}(y_m),\,\,\text{and}\,\,\bigcup\limits_{y_m\in\mathcal{M}_{x_0}}B_{\frac{R}{4}}(y_m)=\bigcup\limits_{l=1}^L\bigcup_{y_m\in\mathcal{A}_{x_0}^l}B_{\frac{R}{4}(y_m)}.$$

Since the covering of $\R^3=\bigcup\limits_{x_i\in\mathcal{I}}B_{\frac{R}{4}}(x_i)$ is given before,   there is a finite set  $\mathcal{I}_{\tilde{x}}$ such that  for any $\tilde{x}\in\R^3$, $B_{R_0}(\tilde{x})\subset\bigcup\limits_{\tilde{x}_i\in\mathcal{I}_{\tilde{x}}}B_{\frac{R}{4}}(\tilde{x}_i)$. In particular,
  $B_{R_0}(z_l)\subset\bigcup\limits_{y_m\in\mathcal{A}_{x_0}^l}B_{\frac{R}{4}}(y_m)$. Then  we have
	 \begin{align*}
	 &\sum_{y_m\in\mathcal{A}_{x_0}^l}\left(R\int_{B_{\frac{R}{4}}(y_m)}|\D v^k|^2+|\D^2u^k|^2dx\right)^\frac{3}{2}\\
	 &\leq \sup_{\tilde{x}\in\R^3}\sum_{\tilde{x}_i\in\mathcal{I}_{\tilde{x}}}\left(R\int_{B_{\frac{R}{4}}(\tilde{x}_i)}|\D v^k|^2+|\D^2u^k|^2dx\right)^\frac{3}{2}
	 \end{align*}
	 which implies
	 \begin{align}\label{4.17}
	 &\sum\limits_{x_i\in \mathcal{I}_{x_0}}\left(R\int_{B_{\frac{R}{2}}(x_i)}(|\D v^k|^2+|\D^2 u^k|^2)dx\right)^\frac{3}{2}\nonumber\\
	 &\leq\sum\limits_{x_i\in \mathcal{I}_{x_0}}\left(\sum_{x_{ij}\in \mathcal{J}_{x_i}}R\int_{B_{\frac{R}{4}}(x_{ij})}(|\D v^k|^2+|\D^2 u^k|^2)dx\right)^\frac{3}{2}\nonumber\\
	 &\leq C|J_{x_i}|^{\frac{3}{2}}\sum_{y_m\in\mathcal{M}_{x_0}}\left(R\int_{B_{\frac{R}{4}}(y_m)}|\D v^k|^2+|\D^2u^k|^2dx\right)^\frac{3}{2}\\
	 &\leq C\sum_{l=1}^L\sum_{y_m\in\mathcal{A}_{x_0}^l}\left(R\int_{B_{\frac{R}{4}}(y_m)}|\D v^k|^2+|\D^2u^k|^2dx\right)^\frac{3}{2}\nonumber\\
	 &\leq CL\sup_{\tilde{x}\in\R^3}\sum_{\tilde{x}_i\in\mathcal{I}_{\tilde{x}}}\left(R\int_{B_{\frac{R}{4}}(\tilde{x}_i)}|\D v^k|^2+|\D^2u^k|^2dx\right)^\frac{3}{2}.\nonumber
	 \end{align}
	 Substituting \eqref{4.17} into \eqref{4.16} and taking sup-norm with respect to $x_0\in\R^3$ and $t$ over $[\delta t_k,t_k]$ give
	 \begin{align}
	 &\sup_{\delta t_k\leq s\leq t_k,x_0\in\R^3}\sum_{x_i\in \mathcal{I}_{x_0}}C\left(R\int_{B_{\frac{R}{4}}(x_i)}|\D v^k(t)|^2+|\D^2 u^k(t)|^2dx\right)^\frac{3}{2}\\
	 &\leq CL\varepsilon_0^2t_k^\frac{3}{2}R^{-3} \sup_{\delta t_k\leq s\leq t_k,x_0\in\R^3}\sum_{x_i\in\mathcal{I}_{x_0}}\left(R\int_{B_{\frac{R}{4}}(x_i)}|\D v^k|^2+|\D^2u^k|^2dx\right)^\frac{3}{2}+\frac{\varepsilon_1^3}{4}\nonumber\\
	 &\leq C_2\varepsilon_0^2 \sup_{\delta t_k\leq s\leq t_k,x_0\in\R^3}\sum_{x_i\in\mathcal{I}_{x_0}}\left(R\int_{B_{\frac{R}{4}}(x_i)}|\D v^k|^2+|\D^2u^k|^2dx\right)^\frac{3}{2}+\frac{\varepsilon_1^3}{4}.\nonumber
	 \end{align}
	 Choosing  $\varepsilon_0$ small enough such that $C_2\varepsilon_0^2\leq\frac{C}{2}$, one has
	 \begin{equation}\label{4.18}
	 \sup_{\delta t_k\leq s\leq t_k, x_0\in\R^3}\sum_{x_i\in \mathcal{I}_{x_0}}C\left(R\int_{B_{\frac{R}{4}}(x_i)}|\D v^k(t)|^2+|\D^2 u^k(t)|^2dx\right)^\frac{3}{2}\leq\frac{\varepsilon_1^3}{2},
	 \end{equation}
	which together with \eqref{4.9} and \eqref{4.4} implies that
	\begin{align*}
	\sup_{x_0\in\R^3}\int_{B_{R_0}(x_0)}|v^k(t_k)|^3+|\D u^k(t_k)|^3dx\leq\frac{3}{4}\varepsilon_1^3.
	\end{align*}
	This contradicts to the assumption that $t_k$ is the maximal time satisfying \eqref{assumption}.
	
	Therefore, there exists a uniform time $T_0=\sigma R_0^2$  and $C=2C_1^\frac{1}{3}$ such that
	\begin{align}\label{3.1}
	\sup_{0\leq t\leq T_0}\|(v^k,\D u^k)\|_{L^3_{R_0}(\R^3)}\leq\varepsilon_1= C\varepsilon_0.
	\end{align}
	Then, by Lemma 3.1, one has, for any $\tau>0$,
		\begin{align}\label{3.5}
		\|(v^k,u^k)\|_{C^m(B_R)\times[\tau,T_0]}\leq C(m,\varepsilon_1,\tau,T_0,R),\quad \text{for any}\quad m\geq 0,R>0.
		\end{align}
	After taking subsequences, \eqref{3.1} and \eqref{3.5} yield that there exist $(v,u)\in C^\infty(\R^3\times(0,T_0],\R^3\times S^2)$ with $(v,\D u)\in L^\infty([0,T_0];L_{uloc}^3(\R^3))$ such that
	\begin{align*}
	&(v^k,\D u^k)\rightharpoonup (v,\D u) \,\,\text{in}\,\, L^3_{loc}(\R^3\times[0,T_0]);\\
	&(v^k,\D u^k)\rightarrow (v,\D u) \,\,\text{in}\,\, C^m(B_R\times[\delta,T_0]),\,\,\forall m\geq 0,R>0,\delta<T_0.
	\end{align*}
	Letting $k\rightarrow +\infty$ in \eqref{3.1} gives
	$$\sup_{0\leq t\leq T_0}\|(v,\D u)\|_{L^3_{R_0}(\R^3)}\leq C\varepsilon_0.$$
	Then, it is easy to check by testing \eqref{O-F-1} and \eqref{O-F-3} by $\varphi$ and $\psi$ in $L^3(0,T_0;W^{1,3}_0(B_R))$ with $\D\cdot\varphi=0$ respectively that
	\begin{equation}\label{3.4}
	\|(\p_t v^k,\p_t u^k)\|_{L^\frac{3}{2}(0,T_0; W^{-1,\frac{3}{2}}(B_R))}\leq C(R),\quad\text{for any}\quad R>0.
	\end{equation}
	Therefore, we have
	$$(v(t),\D u(t))\rightarrow (v_0,\D u_0),\,\,\text{in}\,\,L_{loc}^3(\R^3),\,\,\text{as } t\rightarrow 0,$$
	which implies that $(v,\D u)\in C_{\ast}([0,T_0],L_{uloc}^3(\R^3))$.

\

Now, we prove the characterization of the maximal time $T^\ast$ stated in Theorem \ref{thm1} $(iii)$. Suppose that $T^\ast<+\infty$ and $(iii)$ could not hold true. Then, there exists a $R_\ast$ such that
$$\limsup_{t\rightarrow T^\ast}\|v(t),\D u(t))\|_{L^3_{R_\ast}(\R^3)}\leq \varepsilon_0.$$
In particular, there exists $R_{\star}\in(0,R_\ast]$ such that
$$\sup_{T^\ast-R_\star^2\leq t\leq T^\ast}\|(v(t),\D u(t))\|_{L^3_{R_\star}(\R^3)}\leq \varepsilon_0.$$
Noting that for any $t>0$, $\lim\limits_{|x|\rightarrow\infty}\fint_{B_1(x)}|u^k(y,t)-b|^2dy=0$ and for any $\tau>0$, $u^k\rightarrow u$ in $C^\infty(\R^3\times[\tau,T_0])$, it is obvious that, for any $t>0$,
$\lim\limits_{|x|\rightarrow\infty}\fint_{B_1(x)}|u(y,t)-b|^2dy=0$.
Therefore, by the local existence and high regularity results of Theorem \ref{thm1}, we have $(v,u)\in C^\infty(\R^3\times(0,T^\ast))\cap C_{\ast}([0,T_0],L_{uloc}^3(\R^3))$. Hence, we can extend the solution beyond $T^\ast$, which contradicts to the maximality of $T^\ast$.
\end{proof}

\begin{proof}[\bf Proof of Corollary \ref{cor}:]
For the initial data $(v_0,u_0)$ satisfying the assumption in Corollary \ref{cor}, we have, for any $0<R<\infty$,
$$\sup_{y\in\R^3}\int_{B_R(y)}|v_0|^3+|\D u_0|^3dx\leq \varepsilon_0^3.$$
Then applying Theorem \ref{thm1}, the existence time of solution is at least $\sigma R^2$ which is arbitrarily large. Therefore, we have proved the global existence of solutions with small $L^3$ data of $(v_0,\D u_0)$.
\end{proof}

\section{\bf Uniqueness}
In this section, we prove the uniqueness of $L^3_{uloc}$- solutions to \eqref{O-F-1}-\eqref{O-F-3} with finite initial energy.  Before proving the uniqueness theorem, we show that solutions in Theorem \ref{thm1} have the following a priori estimates under the additional assumption $(v_0,u_0)\in L^2(\R^3;\R^3)\times H_b^1(\R^3,S^2)$.
\begin{lemma}\label{lem5.1}
	Let $(v,u)$ be a solution to \eqref{O-F-1}-\eqref{O-F-3} with initial data $(v_0,u_0)$ satisfying \eqref{Ini} in $\R^3\times(0,T_0)$. If $(v_0,u_0)$ is also in $ L^2(\R^3;\R^3)\times H_b^1(\R^3,S^2)$, then, for any $t\in(0,T_0)$,
	\begin{equation}\label{basic energy}
	\int_{\mathbb R^3}e(u(\cdot,t),v(\cdot,t))dx+\int_{0}^{t}\int_{\mathbb R^3}(|\nabla v|^2+|\partial_tu+(v\cdot\nabla u)|^2dx=\int_{\mathbb{R}^3}e(u_0,v_0)dx,
	\end{equation}
	where $e(u,v)$ is the basic energy defined by $e(u,v)=W(u,\nabla u)+\frac{1}{2}|v|^2$.
\end{lemma}
\begin{proof}
	The equality \eqref{basic energy} follows from multiplying \eqref{O-F-1} and \eqref{O-F-3} by $v^i$ and $\partial_tu^i+v\cdot\nabla u^i$ respectively, summing the resulting equations up and integrating over $\mathbb R^3$. We omit details here, one can refer to Lemma 3.1 of \cite{HX}.
	\end{proof}

\begin{lemma}\label{lem5.2}
Let $(v,u)$ be a solution to \eqref{O-F-1}-\eqref{O-F-3} with initial data $(v_0,u_0)$ satisfying \eqref{Ini} in $\R^3\times(0,T_0)$. If $(v_0,u_0)$ is also in $ L^2(\R^3;\R^3)\times H_b^1(\R^3,S^2)$, then
	\begin{align}
	&\int_{0}^{T_0}\int_{\mathbb{R}^3}|\nabla^2u|^2+|\nabla v|^2dxdt\leq C(1+T_0R_0^{-2})\int_{\mathbb{R}^3}e(u_0,v_0)dx,\label{5.2}\\
	&\int_{0}^{T_0}\int_{\mathbb{R}^3}|\nabla u|^4+|v|^4dxdt\leq C\varepsilon_1^2(1+T_0R_0^{-2})\int_{\mathbb{R}^3}e(u_0,v_0)dx.\label{5.3}
	\end{align}
\end{lemma}

\begin{proof}
	Multiplying \eqref{O-F-1}, \eqref{O-F-3} with $v^i$, $\Delta u^i$ respectively  and using the similar argument in Lemma \ref{key-lem1} (with $\phi\equiv 1$), one has
	\begin{align}\label{5.4}
	&\frac{d}{dt}\int_{\mathbb R^3}|v|^2+|\nabla u|^2dx+2\int_{\mathbb R^3}(|\D v|^2+a|\nabla^2 u|^2)dx\\
	&\leq\int_{\mathbb R^3}(|\D v|^2+a|\nabla^2 u|^2)dx+C\int_{\mathbb R^3}(|\nabla u|^4+|v|^4)dx.\nonumber
	\end{align}
	By a standard covering argument in $\mathbb R^3$, it follows from H\"{o}lder's inequality and \eqref{inter} that
	\begin{align}\label{5.5}
	&\int_{\mathbb R^3}(|\nabla u|^4+|v|^4)dx\leq C\sum_{i}\int_{B_{R_0}(x_i)}(|\nabla u|^4+|v|^4)dx\nonumber\\
	&\leq C\sum_{i}\left(\int_{B_{R_0}(x_i)}|\nabla u|^3+|v|^3\right)^{\frac{2}{3}}\left(\int_{B_{R_0}(x_i)}|\nabla u|^6+|v|^6\right)^\frac{1}{3}\\
	&\leq C\varepsilon_1^2\sum_{i}\left(\int_{B_{R_0}(x_i)}|\nabla^2u|^2+|\nabla v|^2dx+\frac{1}{R_0^2}\int_{B_{R_0}(x_i)}|\nabla u|^2+|v|^2dx\right)\nonumber\\
	&\leq C\varepsilon_1^2(\int_{\mathbb{R}^3}|\nabla ^2u|^2+|\nabla v|^2dx+\frac{1}{R_0^2}\int_{\mathbb{R}^3}|\nabla u|^2+|v|^2dx),\nonumber
	\end{align}
	where we have used that
	\begin{equation*}
	\esssup_{0\leq s\leq T_0,y\in\mathbb R^3}\int_{B_{R_0}(y)}|\nabla u(x,s)|^3+|v(x,s)|^3dx< \varepsilon_1^3.
	\end{equation*}
	Substituting \eqref{5.4} into \eqref{5.5}, integrating over time in $(0,T_0)$ and using \eqref{basic energy}, we prove \eqref{5.2} and \eqref{5.3} by choosing $\varepsilon_1$ small enough.
\end{proof}

 \begin{proof}[\bf Proof of Theorem \ref{thm2}]
 Let $(v_1,u_1)$ and $(v_2,u_2)$ be two solutions of \eqref{O-F-1}-\eqref{O-F-3} with the same initial data satisfying assumptions in Theorem \ref{thm2}. It follows from Theorem \ref{thm1} and Lemmas \ref{lem5.1}, \ref{lem5.2} that
 \begin{equation*}
 \sup_{0\leq t\leq T_0}\|(v_m,\D u_m)\|_{L^3_{R_0}(\R^3)}\leq C\varepsilon_0
 \end{equation*}
 and
 \begin{equation*}
 (v_m,\D u_m)\in L^\infty(0,T_0; L^2(\R^3))\cap L^2(0,T_0; H_b^1(\R^3))\cap L^4(\R^3\times(0,T_0)).
 \end{equation*}
 Then, by testing \eqref{O-F-1} with divergence free vector $\psi$ in $L^2(0,T_0;H^1)$, one has $\p_t v_m\in L^2(0,T_0;H^{-1})$. Define the vector field $\xi_m=(-\Delta+I)^{-1}v_m$, for $m=1,2$, that is, $\xi_m$ is the unique solution to
 \begin{equation}
 -\Delta\xi_m+\xi_m=u_m,\quad\xi_m\rightarrow 0,\quad\text{as}\,x\rightarrow \infty.
 \end{equation}
 It is clear that $\text{div}\,\xi_m=0$, $\xi_m\in L^2(0,T_0;H^3)$ and $\p_t\xi_m\in L^2(0,T_0;H^1).$

 Set $\xi=\xi_1-\xi_2$ and $v=v_1-v_2$. Then, $\p_t\xi$ satisfies
 \begin{align*}
 \p_t\xi&=\p_t(-\Delta+I)^{-1}v\nonumber\\
 &=(-\Delta+I)^{-1}[\text{div}(\D v-v_1\otimes v_1+v_2\otimes v_2-(\D u_1)^TW_p(u_1,\D u_1)\nonumber\\
 &\quad+(\D u_2)^TW_p(u_2,\D u_2))-\D p].
 \end{align*}
Therefore, it follows from integration by parts that
\begin{align}\label{5.7}
&\frac{1}{2}\frac{d}{dt}\int_{\R^3}(|\xi|^2+|\D\xi|^2)dx=\int_{\R^3}(\xi-\Delta\xi)\cdot\p_t\xi dx\nonumber\\
&=\int_{\R^3}(-\Delta+I)\xi\cdot (-\Delta+I)^{-1}[\text{div}(\D v-v_1\otimes v_1+v_2\otimes v_2\nonumber\\
&\quad-(\D u_1)^TW_p(u_1,\D u_1)+(\D u_2)^TW_p(u_2,\D u_2))-\D p]dx\\
&=\int_{\R^3}\xi\cdot[\text{div}(\D v-v_1\otimes v_1+v_2\otimes v_2-(\D u_1)^TW_p(u_1,\D u_1)\nonumber\\
&\quad+(\D u_2)^TW_p(u_2,\D u_2))-\D p]dx\nonumber\\
&=\int_{\R^3}(-\D v+v\otimes v_1+v_2\otimes v):\D\xi dx\nonumber\\
&\quad+\int_{\R^3}((\D u_1)^TW_p(u_1,\D u_1)-(\D u_2)^TW_p(u_2,\D u_2))):\D\xi dx.\nonumber
\end{align}
Since $v=-\Delta\xi+\xi$, one has
\begin{align}\label{5.8}
-\int_{\R^3}\D v:\D \xi dx&=\int_{\R^3}(-\D\xi+\D\Delta\xi):\D\xi dx=-\int_{\R^3}(|\D\xi|^2+|\D^2\xi|^2)dx.
\end{align}
It follows from Young's inequality that
\begin{align}\label{5.9}
&\int_{\R^3}(v\otimes v_1+v_2\otimes v):\D\xi dx\leq \int_{\R^3}(|v_1|+|v_2|)(|\xi|+|\Delta\xi|)|\D\xi|dx\\
&\leq\eta\int_{\R^3}|\D^2\xi|^2dx+C_\eta\int_{\R^3}(1+|v_1|^2+|v_2|^2)(|\xi|^2+|\D\xi|^2)dx.\nonumber
\end{align}
Noting that $W(u,p)$ is quadratic in $u$ and $p$, it is obvious that
\begin{align*}
&(\D u_1)^TW_p(u_1,\D u_1)-(\D u_2)^TW_p(u_2,\D u_2)\\
&\leq C(|\D u_1|^2|w|+(|\D u_1|+|\D u_2|)|\D w|).
\end{align*}
Thus, we obtain
\begin{align}\label{5.10}
&\int_{\R^3}((\D u_1)^TW_p(u_1,\D u_1)-(\D u_2)^TW_p(u_2,\D u_2))):\D\xi dx\nonumber\\
&\leq C\int_{\R^3}(|\D u_1|^2|w|+(|\D u_1|+|\D u_2|)|\D w|)|\D\xi|dx\\
&\leq \eta\int_{\R^3}|\D w|^2dx+C_\eta\int_{\R^3}(|\D u_1|^2|w|^2+(|\D u_1|^2+|\D u_2|^2)|\D\xi|^2)dx.\nonumber
\end{align}
Substituting \eqref{5.8}-\eqref{5.10} into \eqref{5.7} gives
\begin{align}\label{5.11}
&\frac{1}{2}\frac{d}{dt}\int_{\R^3}(|\xi|^2+|\D\xi|^2)dx+\int_{\R^3}(|\D\xi|^2+|\D^2\xi|^2)dx\nonumber\\
&\leq\eta\int_{\R^3}(|\D^2\xi|^2+|\D w|^2)dx+C_\eta\int_{\R^3}|\D u_1|^2|w|^2dx\\
&\quad+C_\eta\int_{\R^3}(1+|v_1|^2+|v_2|^2+|\D u_1|^2+|\D u_2|^2)(|\xi|^2+|\D\xi|^2)dx.\nonumber
\end{align}

On the other hand, set $w=u_1-u_2$. Then, it follows from \eqref{O-F-3} that $w$ satisfies
\begin{align}\label{5.12}
\p_t w^i&=-v\cdot\D u_1^i-v_2\cdot\D w^i+\D_{\alpha}[W_{p_\alpha^i}(u_1,\D u_1)-W_{p_\alpha^i}(u_2,\D u_2)]\nonumber\\
&\quad-\D_{\alpha}[u_1^ku_1^iW_{p_\alpha^k}(u_1,\D u_1)-u_2^ku_2^iW_{p_\alpha^k}(u_2,\D u_2)]-W_{u^i}(u_1,\D u_1)\nonumber\\
&\quad+W_{u^i}(u_2,\D u_2)+u^i_1u^k_1W_{u^k}(u_1,\D u_1)-u^i_2u^k_2W_{u^k}(u_2,\D u_2)\\
&\quad+W_{p_\alpha^k}(u_1,\D u_1)\D_\alpha u^k_1u^i_1-W_{p_\alpha^k}(u_2,\D u_2)\D_\alpha u^k_2u^i_2\nonumber\\
&\quad+W_{p_\alpha^k}(u_1,\D u_1)u^k_1\D_\alpha u^i_1-W_{p_\alpha^k}(u_2,\D u_2)u^k_2\D_\alpha u^i_2.\nonumber
\end{align}
Multiplying \eqref{5.12} by $w^i$ and integrating over $\R^3$ give
\begin{align}\label{5.13}
&\frac{1}{2}\frac{d}{dt}\int_{\R^3}|w|^2dx+\int_{\R^3}\left(W_{p_\alpha^i}(u_1,\D u_1)-W_{p_\alpha^i}(u_2,\D u_2)\right)\D_{\alpha}w^idx\nonumber\\
&=\int_{\R^3}\left(u_1^ku_1^iW_{p_\alpha^k}(u_1,\D u_1)-u_2^ku_2^iW_{p_\alpha^k}(u_2,\D u_2)\right)\D_{\alpha}w^idx\nonumber\\
&\quad-\int_{\R^3}(W_{u^i}(u_1,\D u_1)-W_{u^i}(u_2,\D u_2))w^idx\nonumber\\
&\quad+\int_{\R^3}(u^i_1u^k_1W_{u^k}(u_1,\D u_1)-u^i_2u^k_2W_{u^k}(u_2,\D u_2))w^idx\\
&\quad+\int_{\R^3}(W_{p_\alpha^k}(u_1,\D u_1)\D_\alpha u^k_1u^i_1-W_{p_\alpha^k}(u_2,\D u_2)\D_\alpha u^k_2u^i_2)w^idx\nonumber\\
&\quad+\int_{\R^3}(W_{p_\alpha^k}(u_1,\D u_1)u^k_1\D_\alpha u^i_1-W_{p_\alpha^k}(u_2,\D u_2)u^k_2\D_\alpha u^i_2)w^idx\nonumber\\
&\quad-\int_{\R^3}(v\cdot\D u_1^i+v_2\cdot\D w^i)w^idx\nonumber\\
&=:I_1+I_2+I_3+I_4+I_5+I_6.\nonumber
\end{align}
Since $W(u,\D u):=a_{\alpha\beta}^{ij}(u)\D_{\alpha}u^i\D_{\beta}u^j$ is convex, we have
\begin{equation*}
W_{p_\alpha^i}(u,\D u)=a_{\alpha\beta}^{ij}(u)\D_{\beta}u^j,\quad a_{\alpha\beta}^{ij}(u)\lambda_\alpha^i\lambda_\beta^j\geq a|\lambda|^2,\, \forall \lambda\in \mathbb{M}^{3\times 3},
\end{equation*}
where $a_{\alpha\beta}^{ij}(u)$ is a quadratic polynomial of $u$. Then, we have
\begin{align}\label{5.14}
&\int_{\R^3}\left(W_{p_\alpha^i}(u_1,\D u_1)-W_{p_\alpha^i}(u_2,\D u_2)\right)\D_{\alpha}w^idx\nonumber\\
&=\int_{\R^3}\left((a_{\alpha\beta}^{ij}(u_1)-a_{\alpha\beta}^{ij}(u_2))\D_{\alpha}u_1^j+a_{\alpha\beta}^{ij}(u_2)\D_{\beta}w^j\right)\D_{\alpha}w^idx\\
&\geq a\int_{\R^3}|\D w|^2dx-C\int_{\R^3}|\D u_1||w||\D w|dx,\nonumber
\end{align}
where we have used $|a_{\alpha\beta}^{ij}(u_1)-a_{\alpha\beta}^{ij}(u_2)|\leq C|w|$.
Now we estimate the terms $I_1, \cdots, I_6$ on the right hand side of \eqref{5.13}. Noting that $W(u,\D u)$ is quadratic in $u$ and $\D u$, it follows from tedious but not difficult calculations that
\begin{align}\label{5.15}
\sum_{i=2}^{5}I_i\leq C\int_{\R^3}(|\D u_1|^2|w|+(|\D u_1|+|\D u_2|)|\D w|)|w|dx.
\end{align}
Since $v=-\Delta\xi+\xi$, one has
\begin{equation}\label{5.16}
I_6\leq \int_{\R^3}(|\D u_1|(|\xi|+|\Delta\xi|)+|v_2||\D w|)|w|dx.
\end{equation}
To estimate $I_1$, we first use the rotation invariant property, which will be checked in   Lemma \ref{lem-appen} below, to write the integrand in $I_1$ as
\begin{align}\label{5.17}
&\left(u_1^k(x)u_1^i(x)W_{p_\alpha^k}(u_1,\D_x u_1)-u_2^k(x)u_2^i(x)W_{p_\alpha^k}(u_2,\D_x u_2)\right)\D_{x_\alpha}w^i(x)\\
&=\left(\tilde{u}_1^k(y)\tilde{u}_1^i(y)W_{\tilde{p}_\alpha^k}(\tilde{u}_1,\D_y \tilde{u}_1)-\tilde{u}_2^k(y)\tilde{u}_2^i(y)W_{p_\alpha^k}(\tilde{u}_2,\D_y \tilde{u}_2)\right)\D_{y_\alpha}\tilde{w}^i(y),\nonumber
\end{align}
where $x=Q^Ty$ with $Q\in O(3)$  and $\tilde{u}_m(y)=Qu_m(Q^Ty)$ for $m=1,2$.

 At $x=x_0$, there is  a rotation $Q=Q_{x_0}$ such that $\tilde{u}_2(y_0)=(0,0,1)^T$ with $y_0=Q x_0$.
 Then, for $k=1, 2$, at $y=y_0$, we obtain
\begin{align}\label{5.18}
&\left(\tilde{u}_1^k\tilde{u}_1^iW_{\tilde{p}_\alpha^k}(\tilde{u}_1,\D_y \tilde{u}_1)-\tilde{u}_2^k\tilde{u}_2^iW_{p_\alpha^k}(\tilde{u}_2,\D_y \tilde{u}_2)\right)\D_{y_\alpha}\tilde{w}^i\nonumber\\
&=\tilde{u}_1^k\tilde{u}_1^iW_{\tilde{p}_\alpha^k}(\tilde{u}_1,\D_y \tilde{u}_1)\D_{y_\alpha}\tilde{w}^i=\tilde{w}^k\tilde{u}_1^iW_{\tilde{p}_\alpha^k}(\tilde{u}_1,\D_y \tilde{u}_1)\D_{y_\alpha}\tilde{w}^i\\
&\leq C|\tilde{w}||\D_y\tilde{u}_1||\D_y\tilde{w}|.\nonumber
\end{align}
Similarly, for $i=1$ or $2$, it also holds that at $y=y_0$
\begin{align}\label{5.19}
\left(\tilde{u}_1^k\tilde{u}_1^iW_{\tilde{p}_\alpha^k}(\tilde{u}_1,\D_y \tilde{u}_1)-\tilde{u}_2^k\tilde{u}_2^iW_{p_\alpha^k}(\tilde{u}_2,\D_y \tilde{u}_2)\right)\D_{y_\alpha}\tilde{w}^i\leq C|\tilde{w}||\D_y\tilde{u}_1||\D_y\tilde{w}|.
\end{align}
It suffices to estimate the integrand in $I_1$ for $k=i=3$. Indeed, since $\tilde{u}_2(y_0)=(0,0,1)^T$ and $|\tilde{u}_2|=1$ implies $\D_{_{y_\alpha}}\tilde{u}_2^3=0$ at $y=y_0$, we have that at $y=y_0$
\begin{align}\label{5.20}
&\left(\tilde{u}_1^3\tilde{u}_1^3W_{\tilde{p}_\alpha^3}(\tilde{u}_1,\D_y \tilde{u}_1)-\tilde{u}_2^3\tilde{u}_2^3W_{p_\alpha^3}(\tilde{u}_2,\D_y \tilde{u}_2)\right)\D_{y_\alpha}\tilde{w}^3\nonumber\\
&=\left(\tilde{w}^3\tilde{u}_1^3a_{\alpha\beta}^{ij}(\tilde{u}_1)\D_{y_\beta}\tilde{u}_1^j+\tilde{u}_2^3\tilde{w}^3a_{\alpha\beta}^{ij}(\tilde{u}_1)\D_{y_\beta}\tilde{u}_1^j\right)\D_{y_\alpha}\tilde{w}^3\\
&\quad+\tilde{u}_2^3\tilde{u}_2^3(a_{\alpha\beta}^{ij}(\tilde{u}_1)-a_{\alpha\beta}^{ij}(\tilde{u}_2))\D_{y_\beta}\tilde{u}_1^j\D_{y_\alpha}\tilde{w}^3+\tilde{u}_2^3\tilde{u}_2^3a_{\alpha\beta}^{ij}(\tilde{u}_2)\D_{y_\beta}\tilde{w}^j\D_{y_\alpha}\tilde{u}_1^3\nonumber\\
&\leq C|\D_y\tilde{u}_1||\tilde{w}||\D_y\tilde{w}|+\tilde{u}_2^3\tilde{u}_2^3a_{\alpha\beta}^{ij}(\tilde{u}_2)\D_{y_\beta}\tilde{w}^j\D_{y_\alpha}\tilde{u}_1^3.\nonumber
\end{align}
Since $\tilde{u}_2(y_0)=(0,0,1)^T$ and $|\tilde{u}_1|=1$ , it is clear that at $y=y_0$
\begin{align}\label{5.21}
&\tilde{u}_2^3\tilde{u}_2^3a_{\alpha\beta}^{ij}(\tilde{u}_2)\D_{y_\beta}\tilde{w}^j\D_{y_\alpha}\tilde{u}_1^3\nonumber\\
&=\tilde{u}_2^3\tilde{u}_2^3a_{\alpha\beta}^{ij}(\tilde{u}_2)\D_{y_\beta}\tilde{w}^j(\tilde{u}_1^3\D_{y_\alpha}\tilde{u}_1^3-\tilde{u}_1^3\D_{y_\alpha}\tilde{u}_1^3+\D_{y_\alpha}\tilde{u}_1^3)\nonumber\\
&=\tilde{u}_2^3\tilde{u}_2^3a_{\alpha\beta}^{ij}(\tilde{u}_2)\D_{y_\beta}\tilde{w}^j\tilde{u}_1^3\D_{y_\alpha}\tilde{u}_1^3-\tilde{u}_2^3\tilde{u}_2^3a_{\alpha\beta}^{ij}(\tilde{u}_2)\D_{y_\beta}\tilde{w}^j\tilde{w}^3\D_{y_\alpha}\tilde{u}_1^3\nonumber\\
&\leq-\tilde{u}_2^3\tilde{u}_2^3a_{\alpha\beta}^{ij}(\tilde{u}_2)\D_{y_\beta}\tilde{w}^j(\tilde{u}_1^1\D_{y_\alpha}\tilde{u}_1^1+\tilde{u}_1^2\D_{y_\alpha}\tilde{u}_1^2)+C|\D_y\tilde{u}_1||\tilde{w}||\D_y\tilde{w}|\\
&\leq -\tilde{u}_2^3\tilde{u}_2^3a_{\alpha\beta}^{ij}(\tilde{u}_2)\D_{y_\beta}\tilde{w}^j(\tilde{w}^1\D_{y_\alpha}\tilde{u}_1^1+\tilde{w}^2\D_{y_\alpha}\tilde{u}_1^2)+C|\D_y\tilde{u}_1||\tilde{w}||\D_y\tilde{w}|\nonumber\\
&\leq C|\D_y\tilde{u}_1||\tilde{w}||\D_y\tilde{w}|.\nonumber
\end{align}
Substituting \eqref{5.18}-\eqref{5.21} into \eqref{5.17}, we obtain that
\begin{align}\label{5.22}
&|\left(u_1^ku_1^iW_{p_\alpha^k}(u_1,\D u_1)-u_2^ku_2^iW_{p_\alpha^k}(u_2,\D u_2)\right)\D_{\alpha}w^i| (x_0) \\
&\leq C|\D_y\tilde{u}_1||\tilde{w}||\D_y\tilde{w}|(y_0)=C|\D u_1||w||\D w| (x_0).\nonumber
\end{align}
Since  \eqref{5.22} holds for every $x_0\in\R^3$, we have
\begin{align}\label{5.23}
I_1\leq C\int_{\R^3}|\D u_1||w||\D w|dx.
\end{align}
Substituting \eqref{5.14}-\eqref{5.16} and \eqref{5.23} into \eqref{5.13}, we have
\begin{align}\label{5.24}
&\frac{1}{2}\frac{d}{dt}\int_{\R^3}|w|^2dx+a\int_{\R^3}|\D w|^2dx\\
&\leq C\int_{\R^3}|\D u_1|^2|w|^2+(|v_2|+|\D u_1|+|\D u_2|)|w||\D w|+|\D u_1|(|\xi|+|\Delta\xi|)|w|dx\nonumber\\
&\leq \eta\int_{\R^3}(|\D^2\xi|^2+|\D w|^2)dx+C_\eta\int_{\R^3}(1+|v_2|^2+|\D u_1|^2+|\D u_2|^2)(|\xi|^2+|w|^2)dx.\nonumber
\end{align}
Summing \eqref{5.11} with \eqref{5.24} and choosing $\eta=\frac{1}{2}\min\{1,a\}$ give
\begin{align}\label{5.25}
&\frac{d}{dt}\int_{\R^3}(|\xi|^2+|\D\xi|^2+|w|^2)dx+\int_{\R^3}(|\D\xi|^2+|\D^2\xi|^2+a|\D w|^2)dx\\
&\leq C\int_{\R^3}(1+|v_1|^2+|v_2|^2+|\D u_1|^2+|\D u_2|^2)(|\xi|^2+|\D\xi|^2+|w|^2)dx.\nonumber
\end{align}
It follows from a covering argument and \eqref{inter} that
\begin{align}\label{5.26}
&\int_{\R^3}(|v_1|^2+|v_2|^2+|\D u_1|^2+|\D u_2|^2)(|\xi|^2+|\D\xi|^2+|w|^2)dx\\
&\leq C\sum_{i}\int_{B_{R_0}(x_i)}(|v_1|^2+|v_2|^2+|\D u_1|^2+|\D u_2|^2)(|\xi|^2+|\D\xi|^2+|w|^2)dx\nonumber\\
&\leq C\sum_{i}\left(\int_{B_{R_0}(x_i)}|v_1|^3+|v_2|^3+|\D u_1|^3+|\D u_2|^3\right)^\frac{2}{3}\left(\int_{B_{R_0}(x_i)}|\xi|^6+|\D\xi|^6+|w|^6\right)^\frac{1}{3}dx\nonumber\\
&\leq C\varepsilon_1^2\sum_{i}\left(\int_{B_{R_0}(x_i)}|\D\xi|^2+|\D^2\xi|^2+|\D w|^2dx+\frac{1}{R_0^2}\int_{B_{R_0}(x_i)}|\xi|^2+|\D\xi|^2+|w|^2dx\right)\nonumber\\
&\leq C\varepsilon_1^2\int_{\R^3}|\D\xi|^2+|\D^2\xi|^2+|\D w|^2dx+\frac{C\varepsilon_1^2}{R_0^2}\int_{\R^3}(|\xi|^2+|\D\xi|^2+|w|^2)dx.\nonumber
\end{align}
Collecting \eqref{5.25} and \eqref{5.26}, we obtain
\begin{align}
\frac{d}{dt}(\|\xi(t)\|_{L^2}^2+\|\D\xi(t)\|_{L^2}^2+\|w(t)\|_{L^2}^2)\leq C(\|\xi\|_{L^2}^2+\|\D\xi\|_{L^2}^2+\|w\|_{L^2}^2)
\end{align}
which implies that
\begin{align*}
\|\xi(t)\|_{L^2}^2+\|\D\xi(t)\|_{L^2}^2+\|w(t)\|_{L^2}^2\leq(\|\xi(0)\|_{L^2}^2+\|\D\xi(0)\|_{L^2}^2+\|w(0)\|_{L^2}^2) e^{Ct}\leq 0.
\end{align*}
Therefore, $\xi=\D\xi=w=0$ which completes a proof of uniqueness.
 \end{proof}

During the proof, we used the following elementary lemma, which is based on the rotation invariant property $W(Q u,Q\D u Q^T)=W(u,\D u)$  for a rotation $Q\in O(3)$(see \cite{Ga}):
\begin{lemma}\label{lem-appen}
	For a rotation $Q\in O(3)$, the term $u^iu^kW_{p^k_\alpha}(u,\D u)\D_{\alpha}w$ has the following invariant property:
	\begin{align*}
	u^i(x)u^k(x)W_{p^k_\alpha}(u(x),\D_x u(x))\D_{x_\alpha}w(x)^i=\tilde{u}^i(y)\tilde{u}^k(y)W_{\tilde{p}^k_\alpha}(\tilde{u}(y),\D_y\tilde{u}(y))\D_{y_\alpha}\tilde{w}^i(y),
	\end{align*}
	where $x=Q^Ty$ and $(\tilde{u}(y),\tilde{w}(y))=(Qu(Q^Ty),Qw(Q^Ty))$.
\end{lemma}
\begin{proof}
	Let $x=Q^Ty$ and $\tilde{u}(y)=Qu(Q^Ty)$.  By the chain rule, one has
	\begin{equation*}
	\D_{x_\beta} u^l=(Q^T)_{lk}\D_{y_\alpha}\tilde{u}^kQ_{\alpha\beta}.
	\end{equation*}
	Using this fact, it is easy to check that $W(\tilde{u},\D_y\tilde{u})=W(u,\D u)$ (see \cite{Ga}). Then,
	\begin{align*}
	&W_{\tilde{p}_\alpha^k}(\tilde{u}(y),\D_y \tilde{u}(y))=W_{\tilde{p}_\alpha^k}(u(x),u(x))\nonumber\\
	&=W_{p_\beta^l}(u(x),\D u(x))\frac{\p\D_{x_\beta}u^l}{\p\D_{_{y_\alpha}} \tilde{u}^k}=W_{p_\beta^l}(u(x),\D u(x))(Q^T)_{lk}Q_{\alpha\beta}.
	\end{align*}
	Therefore, we have
	\begin{align*}
	&\tilde{u}^i(y)\tilde{u}^k(y)W_{\tilde{p}^k_\alpha}(\tilde{u},\D_y\tilde{u})\D_{y_\alpha}\tilde{w}(y)\nonumber\\
	&=Q_{ij}u^j(x)Q_{km}u^m(x)W_{p_\beta^l}(u(x),\D u(x))(Q^T)_{lk}Q_{\alpha\beta}Q_{in}\D_{x_\g}u^n(x)(Q^T)_{\g\alpha}\\
	&=\delta_{nj}\delta_{lm}\delta_{\beta\g}u^j(x)u^m(x)W_{p_\beta^l}(u(x),\D u(x))\D_{x_\g}u^n(x)\nonumber\\
	&=u^j(x)u^l(x)W_{p_\beta^l}(u(x),\D u(x))\D_{x_\beta}u^j(x).\nonumber
	\end{align*}
	\end{proof}

\noindent{\bf Acknowledgments} The research  of the first
author was supported by the Australian Research Council
grant DP150101275. The second author was supported by the Australian Research Council
grant DP150101275 as an postdoctoral fellow.

\end{document}